\documentclass[12pt]{amsart}

\usepackage{amsfonts,amsmath,amssymb,amsxtra,amsthm,tikz,combelow}
\usepackage{mathrsfs}
\usepackage{times}
\usepackage{anysize}
\usepackage{graphicx}
\marginsize{1in}{1in}{1in}{1in}

\newtheorem{theorem}{Theorem}[section]
\newtheorem{lemma}[theorem]{Lemma}

\theoremstyle{definition}
\newtheorem{definition}[theorem]{Definition}

\newtheorem{conjecture}[theorem]{Conjecture}
\newtheorem{corollary}[theorem]{Corollary}
\newtheorem{proposition}[theorem]{Proposition}

\theoremstyle{remark}
\newtheorem{remark}[theorem]{Remark}

\DeclareMathOperator{\ch}{ch}
\DeclareMathOperator{\gr}{gr}
\DeclareMathOperator{\Tr}{Tr}

\DeclareMathOperator{\Hilb}{Hilb}

\DeclareMathOperator{\Fr}{Fr}
\DeclareMathOperator{\Fl}{Fl}
\DeclareMathOperator{\Gr}{Gr}

\newcommand{\Hl}{\widetilde{H}_{\lambda}}
\newcommand{\Hlt}{\widetilde{H}_{\lambda^t}}

\newcommand{\Vv}{\mathbf{v}}

\newcommand{\CP}{\mathcal{P}}
\newcommand{\CO}{\mathcal{O}}
\newcommand{\CT}{\mathcal{T}}
\newcommand{\CI}{\mathcal{I}}
\newcommand{\CL}{\mathcal{L}}
\newcommand{\CF}{\mathcal{F}}
\newcommand{\CW}{\mathcal{W}}

\newcommand{\la}{\lambda}

\newcommand{\BC}{\mathbb{C}}
\newcommand{\BK}{\mathbb{K}}
\newcommand{\BP}{\mathbb{P}}
\newcommand{\BZ}{\mathbb{Z}}

\newcommand{\h}{\mathbf{H}}
\newcommand{\sh}{\mathbf{SH}}
\newcommand{\slz}{SL_2(\BZ)}
\newcommand{\syt}{\textrm{SYT}}

\newcommand{\e}{\varepsilon}

\newcommand{\sym}{\textrm{Sym}}
\newcommand{\res}{\textrm{Res}}
\newcommand{\Tes}{\textrm{Tes}}
\newcommand{\Des}{\textrm{Des}}
\newcommand{\dinv}{\textrm{dinv}}
\newcommand{\qTes}{\textrm{qTes}}
\newcommand{\sq}{\square}

\newcommand{\pf}{\text{PF}}
\newcommand{\ph}{\varphi_{\frac 1{1-1/t}}}
\newcommand{\phh}{\varphi^{-1}_{\frac 1{1-1/t}}}

\newcommand{\TP}{\widetilde{P}}
\newcommand{\TCP}{\widetilde{\CP}}

\title{Refined knot invariants and Hilbert schemes}

\author{Eugene Gorsky}
\address {Department of Mathematics, Columbia University, 2990 Broadway \\ New York, NY 10027 }
\address {Department of Mathematics, UC Davis, One Shields Ave\\ Davis, CA 95616}
\address {International Laboratory of Representation Theory and Mathematical Physics\\
NRU-HSE, 7 Vavilova St.\\ 
Moscow, Russia 117312}
\email{egorsky@math.columbia.edu}

\author{Andrei Negu\cb t}
\address{Mathematics Department, Columbia University, New York, NY 10027, USA}
\address{Simion Stoilow Institute of Mathematics, Bucharest, Romania}
\email{andrei.negut@gmail.com}

\begin{document}

\begin{abstract}
We consider the construction of refined Chern-Simons torus knot invariants by M. Aganagic and S. Shakirov from the DAHA viewpoint of I. Cherednik. We give a proof of Cherednik's conjecture on the stabilization of superpolynomials, and then use the results of O. Schiffmann and E. Vasserot to relate knot invariants to the Hilbert scheme of points on $\BC^2$. Then we use the methods of the second author to compute these invariants explicitly in the uncolored case. We also propose a conjecture relating these constructions to the rational Cherednik algebra, as in the work of the first author, A. Oblomkov, J. Rasmussen and V. Shende. Among the combinatorial consequences of this work is a statement of the $\frac mn$ shuffle conjecture.
\end{abstract}

\maketitle

%\tableofcontents

\section{Introduction}

In \cite{AS}, Aganagic and Shakirov defined refined invariants of the $(m,n)$ torus knot by constructing two matrices $S$ and $T$ that act on an appropriate quotient of the Fock space, and satisfy the relations in the group $\slz$ (as in the work of Etingof and Kirillov). They conjectured that these invariants match the Poincar\'e polynomials of Khovanov-Rozansky HOMFLY homology (\cite{KhR2,KhSoergel}) of torus knots. The computation of Khovanov-Rozansky homology for torus knots is a hard open problem in knot theory, and the Aganagic-Shakirov conjecture has been verified in a few cases when this homology can be explicitly computed from the definition.

In \cite{Ch}, Cherednik reinterpreted the construction of \cite{AS} in terms of the spherical double affine Hecke algebra $\sh$ of type $A$, by replacing the $\slz$ action on the Fock space representation by an $\slz$ action on the DAHA. As such, he obtained a conjectural definition for the three variable torus knot invariant known as the \textbf{superpolynomial} $\CP_{n,m}^\lambda(u,q,t)$, defined for all partitions $\lambda$ and all pairs of coprime integers $m,n$. In Cherednik's viewpoint, these superpolynomials arise as evaluations of certain elements $P_{n,m}^\lambda$ in the DAHA. These elements are polynomials in:
$$
P_{kn,km} \in \sh, \qquad \forall k\in \BZ,
$$ 
by the same formula as the well-known Macdonald polynomials $P_\lambda$ are polynomials in the power sum functions $p_k$, for all partitions $\lambda$. The DAHA $\sh$ is bigraded in such a way that $P_{kn,km}$ lies in bidegrees $(kn,km)$, and the action of $\slz$ on $\sh$ is by automorphisms which permute the bidegrees of the elements $P_{kn,km}$. Cherednik constructs this action by a sequence of elementary transformations, which are however rather difficult to describe by a closed formula. %It was conjectured in \cite{Ch} that the $N$-dependence can be eliminated by the introduction of a new variable $u$ and specialization $u=t^{N}$.

Schiffmann and Vasserot (\cite{svhilb,svmacd}) give an alternate description of the DAHA by showing that it is isomorphic to the elliptic Hall algebra. Under this automorphism, $P_{kn,km}$ correspond to the standard generators of the elliptic Hall algebra described by Burban and Schiffmann in \cite{BS}. Moreover, Schiffmann and Vasserot show that $\sh$ acts on the $K-$theory of the Hilbert scheme of points on the plane, and we will show that Cherednik's superpolynomials can be computed in this representation. We use this viewpoint to prove two conjectures announced by Cherednik in \cite{Ch} (let us remark that while the present paper was being written, Cherednik also announced independent proofs in \cite[Section 2.4.1]{Ch}). The occurrence of the Hilbert scheme is not so surprising: Nakajima (\cite{Nak}) already used it to compute the matrices $S$ and $T$ of Aganagic and Shakirov, although by using a different construction from ours. 

An explicit description of the action of $P_{kn,km}$ on the $K-$theory of the Hilbert scheme was obtained in \cite{Negut}, where these operators were shown to be described by a certain geometric correspondence called the flag Hilbert scheme. We believe that a certain line bundle on this moduli space is related to the unique finite-dimensional irreducible module $L_{\frac mn}$ of the rational Cherednik algebra (see Conjecture \ref{conj:big} for a precise statement). A computational consequence of our approach is the following formula for uncolored superpolynomials as a sum over standard Young tableaux.

\begin{theorem}
\label{thm:main}
The superpolynomial $\CP_{n,m}(u,q,t)$, defined as in \cite{Ch}, is given by:
\begin{equation}
\label{eqn:for}
\CP_{n,m}(u,q,t) =\sum_{\mu \vdash n} \frac {\widetilde{\gamma}^n}{\widetilde{g}_\mu} \sum^{\syt}_{\text{of shape }\mu}  \frac {\prod_{i=1}^{n} \chi_i^{S_{m/n}(i)} (1-u\chi_i)(q \chi_i - t)}{\left(1-\frac {q\chi_{2}}{t\chi_{1}}\right) \ldots  \left(1-\frac {q\chi_{n}}{t\chi_{n-1}}\right)}\prod_{1\leq i < j\leq n} \frac {(\chi_j-q\chi_i)(t\chi_j-\chi_i)}{(\chi_j-\chi_i)(t\chi_j-q\chi_i)}
\end{equation}
where the sum is over all standard Young tableaux of size $n$, and $\chi_i$ denotes the $q,t^{-1}$-weight of the box labeled by $i$ in the tableau. The constants in the above relation are given by:
\begin{equation}
\label{def:smn}
S_{m/n}(i)=\left\lfloor \frac {im}n\right\rfloor - \left\lfloor \frac {(i-1)m}n \right\rfloor, \qquad \qquad \widetilde{\gamma}=\frac {(t-1)(q-1)}{(q-t)}
\end{equation}
and:
$$
\widetilde{g}_\mu =\prod_{\square \in \lambda}(1-q^{a(\square)}t^{l(\square)+1}) \prod_{\square \in \lambda}(1-q^{-a(\square)-1}t^{-l(\square)})
$$
The notions of arm-length $a(\square)$ and leg-length $l(\square)$ of a box in a Young diagram will be recalled in Figure \ref{fig}.
\end{theorem}

We also give a prescription to compute general colored superpolynomials $\CP_{n,m}^\lambda$, for example on a computer, although we do not yet have any ``nice'' formula. We use formulas such as \eqref{eqn:for} to explore many combinatorial consequences, such as to prove or formulate conjectures about $q,t-$Catalan numbers, parking functions and Tesler matrices in Section \ref{sec:comb}. The highlight is an $\frac mn$ version of the shuffle conjecture \cite{HHRLU}, where a certain combinatorial sum over parking functions in an $m \times n$ rectangle obtained by Hikita \cite{Hikita} is connected with the operators $P_{n,m}$ (see Conjecture \ref{conj:shuffle} for all details). A sample of these results is a corollary of Conjecture \ref{conj:shuffle}, which appears to be new and interesting by itself:   

%The present paper is connected to certain results in mathematical physics, which we will now review. In \cite{HLRV1,HLRV2} Hausel, Letellier and Rodriguez-Villegas studied certain character varieties and conjectured in \cite[Conjecture 1.2.1]{HLRV1} an explicit formula for the generating function of their Hodge polynomials.  In \cite{dCHM} de Cataldo, Hausel and Migliorini conjectured that the Hodge filtration in the cohomology of the character variety matches the perverse filtration in the cohomology of the corresponding Hitchin moduli space and proved this conjecture in $A_1$ case. 
%In \cite{CDP} Chuang, Diaconescu and Pan interpreted the conjectural formula from \cite{HLRV1} as a certain K-theory integral over the Hilbert scheme of points. They also argued that the conjectures of \cite{HLRV1,HLRV2} belong to the physical paradigm of the %``geometric engineering'' relating the integrals  over the Hilbert scheme of points (and, more generally, over the moduli spaces of framed sheaves on $\mathbb{P}^1$)
%to certain string theory partition functions (see e.g. \cite{Nek}).

\begin{conjecture}
\label{conj:super}
The ``superpolynomial'' $\CP_{n,m}(u,q,t)$ can be written as a following sum:
$$
\CP_{n,m}(u,q,t) = \sum_{D} q^{\delta_{m,n}-|D|}t^{-h_{+}(D)}\prod_{P\in v(D)}(1-ut^{\beta(P)}) 
$$
Here the summation is over all lattice paths $D$ contained below the main diagonal of an $m\times n$l, $P$ goes over all vertices of $D$, $\delta_{m,n}=\frac{(m-1)(n-1)}{2}$ and
$h_{+}(D)$ and $\beta(P)$ are certain combinatorial statistics  (see Section \ref{sec:catalan} for details).
\end{conjecture}

For $m=n+1$, the above identity was conjectured in \cite{EHKK} and proved in \cite{HSchroeder}. In the present paper, we prove this conjecture in the limit $t=1$.
For many values of $m$ and $n$, the above conjecture has been verified on a computer. 

We also note that Conjecture \ref{conj:super} can be regarded as a refinement of \cite[Conjectures 23,24]{ORS}. Indeed, in \cite{ORS} the authors conjectured
a relation between the Khovanov-Rozansky homology of an algebraic knot and the homology of the Hilbert schemes of points on the corresponding plane curve singularity.
For torus knots, these Hilbert schemes admit pavings by affine cells, and the homology can be computed combinatorially by counting these cells weighted with their dimensions.
It has been remarked in \cite[Appendix A.3]{ORS} that the corresponding combinatorial sum can be rewritten as a sum over lattice paths matching the combinatorial side
of Conjecture \ref{conj:super}. Furthermore, it has been conjectured in \cite[Conjectures 23,24]{ORS}
that the Poincar\'e polynomial for the Khovanov-Rozansky homology of torus knots can be rewritten as an equivariant character of the space of sections of a certain sheaf
on the Hilbert scheme of points on $\BC^2$, and this sheaf was written explicitly for $m=kn\pm 1$. For general $m$, a construction of such a sheaf or its class in the equivariant $K$-theory was not accessible by the methods of \cite{ORS} 
(see \cite[p. 24]{ORS} for the extensive computations).

In the present paper, we present a (conjectural) candidate for such a sheaf for all $m$ coprime with $n$, using the geometric realization of $\CP_{n,m}(u,q,t)$.
To summarize, one can say that Conjecture \ref{conj:super} would imply the agreement between the Aganagic-Shakirov and Oblomkov-Rasmussen-Shende conjectural descriptions of HOMFLY homology of torus knots,
though, indeed, the relation between both of these descriptions to the actual definition of \cite{KhR2,KhSoergel} remains unknown. 

Another conjecture relates the operator $P_{m,n}$ to the finite-dimensional representation $L_{\frac mn}$ of the rational Cherednik algebra with parameter $c=\frac{m}{n}$
equipped with a certain filtration defined in \cite{GORS}. Such a representation is naturally graded and carries an action of symmetric group $S_n$ preserving both the grading and the filtration. 

\begin{conjecture}
The bigraded Frobenius character of $L_{\frac mn}$, equipped with the natural grading and extra filtration (defined in \cite{GORS}) equals $P_{m,n}\cdot 1$.
\end{conjecture}

This conjecture has bee mainly motivated by \cite{GORS} where $L_{\frac mn}$  has been related to the Hilbert schemes on the
singular curve $\{x^m=y^n\}$ and to the knot homology. For $m=kn\pm 1$, the conjecture follows from the results of \cite{GS,GS2}.

%It was explained in \cite{GORS} that the compactified Jacobian of the singularity $\{x^m=y^n\}$ and its parabolic version can be interpreted as local models for the fibers in the Hitchin fibration. They admit affine cell decompositions, and the combinatorial statistics coincide with the dimension of a cell and, conjecturally, with the perverse filtration. Therefore, the identities of Section \ref{sec:comb} can be interpreted as relations between the perverse Poincar\'e polynomials of the fibers of the Hitchin map and integrals over the Hilbert scheme, which is very similar to the above. In the physical framework of \cite{AS} it was also argued that the parameters $q$ and $t$ in refined Chern-Simons theory are related via ``geometric engineering'' to the equivariant parameters on the Hilbert scheme side. We plan to pursue the geometric meaning of the identities in Section \ref{sec:comb} in our future work.

The structure of this paper is the following: in Section \ref{sec:sym}, we recall the basics on symmetric functions, Macdonald polynomials, the double affine Hecke algebra, we state Cherednik's conjectures \ref{conj1} and \ref{conj2}, and recall how they relate to the original construction of Aganagic and Shakirov in Chern-Simons theory. In Section \ref{sec:stab}, we use the stabilization procedure of Schiffmann-Vasserot to prove Cherednik's conjectures by recasting his superpolynomials as matrix coefficients. In Section \ref{sec:hilb}, we discuss the Hilbert scheme and the flag Hilbert scheme, and show how the machinery of \cite{Negut} gives new formulas for torus knot invariants. In Section \ref{sec:rat}, we discuss the connection between the Hilbert scheme and the rational Cherednik algebra, and conjecture that a certain line bundle on the flag Hilbert scheme corresponds to the representation $L_{\frac mn}$ under the Gordon-Stafford functor. Finally, in Section \ref{sec:comb}, we present certain aspects from the combinatorics of symmetric functions which arise in connection to our work, state several conjectures about $q,t-$Catalan numbers, parking functions and Tesler matrices, which we prove in several special cases.

\section*{Acknowledgments}
We are deeply grateful to Andrei Okounkov, who spurred this paper by explaining to us the relation between the earlier results of the second author to the work of the first author on knot invariants. He has helped us greatly with a lot of advice, and many interesting and educating discussions. We are grateful to M. Aganagic, F. Bergeron, R. Bezrukavnikov, I. Cherednik, P. Etingof, A. Garsia, I. Gordon, J. Haglund, M. Haiman,  A. Kirillov Jr., I. Losev, M. Mazin, H. Nakajima, N. Nekrasov, A. Oblomkov, J. Rasmussen,  S. Shadrin, S. Shakirov, A. Smirnov, A. Sleptsov and V. Shende for their interest and many useful discussions. Special thanks to Maxim Kazaryan for helping us with the {\tt Mathematica} code. The work of E. G. was partially supported by the NSF grant DMS-1403560, RFBR grants RFBR-10-01-678, RFBR-13-01-00755  and the Simons foundation.

\section{DAHA and Macdonald polynomials}
\label{sec:sym}

\subsection{Symmetric functions}
\label{sub:symfunc}

Consider formal parameters $q$ and $t$, and let us define the following ring of constants and its field of fractions:
\begin{equation}
\label{eqn:rings}
\BK_0 = \BC[q^{\pm 1},t^{\pm 1}], \qquad \qquad \BK = \BC(q,t) 
\end{equation}
Among our basic objects of study will be the algebras of symmetric polynomials:
$$
V = \BK[x_1,x_2,\ldots]^\sym, \qquad \qquad V_N = \BK[x_1,\ldots,x_N]^\sym
$$
In fact, the algebras $V_N$ form a projective system $V_N \longrightarrow V_{N-1}$, with the maps given by setting $x_N=0$, and the inverse limit of the system is $V$. An important system of generators for these vector spaces consists of power-sum functions $p_k = \sum_i x_i^k$, for which we have:
$$
V = \BK[p_1,p_2,\ldots], \qquad \qquad V_N = \BK[p_1,\ldots,p_N]
$$
A linear basis of $V$ is given by: 
$$
p_\lambda = p_{\lambda_1} p_{\lambda_2}\ldots, \qquad \text{as} \quad \lambda = (\lambda_1 \geq \lambda_2\geq\ldots)
$$
go over all integer partitions. We have the scalar product $\langle \cdot, \cdot \rangle$ on $V$ given by:
\begin{equation}
\label{eqn:innerp}
\langle p_{\lambda},p_{\mu}\rangle = \delta_{\lambda}^{\mu} z_{\lambda} \qquad \forall \ \lambda,\mu
\end{equation}
where $z_{1^{n_1}2^{n_2}\ldots} = \prod_{i\geq 1} i^{n_i} n_i!$. Then another very important basis of $V$ is given by the Schur functions $s_\lambda$, which are orthogonal under the above scalar product and satisfy:

$$
s_\lambda = m_\lambda + \sum_{\mu < \lambda} c_\lambda^\mu m_\mu \qquad \textrm{for some } c^\mu_\lambda \in \BZ
$$
where $m_\lambda = \sym \left(z_1^{\lambda_1}z_2^{\lambda_2}\ldots \right)$ are the monomial symmetric functions, and $<$ denotes the dominance partial ordering on partitions: $\mu \leq \lambda$ if $\mu_1+\ldots+\mu_i \leq \lambda_1+\ldots+\lambda_i$ for all $i\geq 1$.

\subsection{Macdonald polynomials} 

Another remarkable inner product on $V$ was introduced by Macdonald \cite{Macdonald}:
\begin{equation}
\label{eqn:innerpqt}
\langle p_{\lambda},p_{\mu}\rangle_{q,t}= \delta_{\lambda}^{\mu} z_{\lambda}\prod_{i}\frac{1-q^{\lambda_{i}}}{1-t^{\lambda_{i}}} \qquad \forall \ \lambda,\mu
\end{equation}
The Macdonald polynomials $P_\lambda$ are defined by the property of being orthogonal with respect to $\langle \cdot, \cdot \rangle_{q,t}$ and upper triangular in the basis of monomial symmetric functions:
$$
P_\lambda = m_\lambda + \sum_{\mu < \lambda} d^\mu_\lambda m_\mu \qquad \textrm{for some } d^\mu_\lambda \in \BK
$$
The square norm of $P_{\lambda}$ is given by:
$$
\langle P_{\lambda},P_{\lambda}\rangle_{q,t} = \frac{h'_{\lambda}}{h_{\lambda}},
$$
where:

$$
h_{\lambda} = \prod_{\square \in \lambda}(1-q^{a(\square)}t^{l(\square)+1}), \qquad  h'_{\lambda}=\prod_{\square \in \lambda}(1-q^{a(\square)+1}t^{l(\square)})
$$
Here $\square$ goes over all the boxes in the Young diagram associated to the partition $\lambda$. The arm-length $a(\square)$ (respectively, the leg-length $l(\square)$) is defined as the number of boxes above (respectively, to the right) of the box $\square$. For illustration, see Figure \ref{fig}.

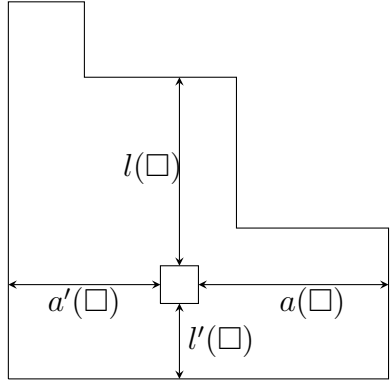
\begin{figure}[ht]
\begin{tikzpicture}
\draw (0,0)--(0,5)--(1,5)--(1,4)--(3,4)--(3,2)--(5,2)--(5,0)--(0,0);
\draw (2,1)--(2,1.5)--(2.5,1.5)--(2.5,1)--(2,1);
\draw (2.25,1.25) node {};
\draw [<->,>=stealth] (2.5,1.25)--(5,1.25);
\draw [<->,>=stealth] (0,1.25)--(2,1.25);
\draw [<->,>=stealth] (2.25,1.5)--(2.25,4);
\draw [<->,>=stealth] (2.25,0)--(2.25,1);
\draw (4,1) node {$a(\square)$};
\draw (1,1) node {$a'(\square)$};
\draw (1.9,2.8) node {$l(\square)$};
\draw (2.8,0.5) node {$l'(\square)$};
\end{tikzpicture}
\caption{Arm, leg, co-arm and co-leg}
\label{fig}
\end{figure}

%the Young diagram of the partition $(4,3,1)$ is displayed below. The box marked by $*$ has $a(\square) = 1, l(\square)=2$:

%\begin{figure}[ht]
%\begin{picture}(100,160)(50,-15)
 
%\put(0,0){\line(1,0){160}}
%\put(0,40){\line(1,0){160}}
%\put(0,80){\line(1,0){120}}
%\put(0,120){\line(1,0){40}}

%\put(0,0){\line(0,1){120}}
%\put(40,0){\line(0,1){120}}
%\put(80,0){\line(0,1){80}}
%\put(120,0){\line(0,1){80}}
%\put(160,0){\line(0,1){40}}

%\put(160,40){\circle*{5}}
%\put(120,80){\circle*{5}}
%\put(40,120){\circle*{5}}

%\put(160,0){\circle{5}}
%\put(120,40){\circle{5}}
%\put(40,80){\circle{5}}
%\put(0,120){\circle{5}}

%\put(57,17){$*$}

%\put(162,3){\scriptsize{$(4,0)$}}
%\put(162,43){\scriptsize{$(4,1)$}}
%\put(122,43){\scriptsize{$(3,1)$}}
%\put(122,83){\scriptsize{$(3,2)$}}
%\put(42,83){\scriptsize{$(1,2)$}}
%\put(42,123){\scriptsize{$(1,3)$}}
%\put(2,123){\scriptsize{$(0,3)$}}

%\end{picture}
%\caption{Notations for a Young diagram}
%\label{fig}
%\end{figure}

There are various normalizations of Macdonald polynomials, and we will encounter their integral form $J_{\lambda}=h_{\lambda}P_{\lambda}$. We have:
\begin{equation}
\label{eqn:normj}
\langle J_{\lambda},J_{\lambda}\rangle_{q,t} = h_{\lambda}^2\langle P_{\lambda},P_{\lambda}\rangle_{q,t} =h_{\lambda} h'_{\lambda}.
\end{equation}
All these constructs make sense both in $V$ (infinitely many variables) and in $V_N$ (finitely many variables), and they are compatible under the maps $V \longrightarrow V_N$. Let us now focus on the case of finitely many variables. Fix a positive integer $N$ and let: 
$$
\rho_N=(N-1,N-2,\ldots,1,0)
$$ 
Consider the evaluation homomorphism:
\begin{equation}
\label{eqn:eval}
\e_N: V_N \longrightarrow \BK, \qquad \e_N(f) = f(t^{\rho}) = f(t^{N-1},t^{N-2},\ldots,1)
\end{equation}
A simple computation reveals that:
\begin{equation}
\label{power sum eval}
\e_N(p_{k}) = \frac{1-t^{kN}}{1-t^k}=\frac{1-u^{k}}{1-t^{k}},
\end{equation}
where we capture the $N$-dependence in the new variable $u=t^{N}$. Therefore, one can compute $\e_N(f)$ for a general symmetric function $f$ by expanding it in terms of the $p_{k}$ and then using \eqref{power sum eval}.

\begin{theorem}(\cite[eq. VI.6.17 and VI.8.8]{Macdonald})
\label{thm:jeval} 
The following equations hold:
\begin{equation}
\label{J eval}
\e_N(P_{\lambda}) = \frac 1{h_\lambda} \prod_{\square\in \lambda} (t^{l'(\square)}-uq^{a'(\square)})
 \qquad \Longrightarrow \qquad \e_N(J_{\lambda}) = \prod_{\square\in \lambda}\left(t^{l'(\square)}-uq^{a'(\square)}\right)
\end{equation}
where $a'(\square)$ (respectively, $l'(\square)$) denote the co-arm and co-leg of $\square$ in $\lambda$ (see Figure \ref{fig}).
\end{theorem}
 
\subsection{Double affine Hecke algebras}
 
Following \cite{ChBook}, we will define the double affine Hecke algebra (DAHA) of type $A_N$.

\begin{definition}
The algebra $\mathbf{H}_{N}$ is defined over $\BK$ by generators $T_i^{\pm 1}$ for $i \in \{1,\ldots,N-1\}$, and ${X_j}^{\pm 1}, {Y_j}^{\pm 1}$ for  $j\in \{1,\ldots,N\}$, under the following relations:

\begin{multline}
\label{DAHA relations}
(T_i+t^{-1/2})(T_i-t^{1/2})=0,\qquad T_{i}T_{i+1}T_{i}=T_{i+1}T_{i}T_{i+1},\qquad [T_i,T_k]=0\ \mbox{for}\ |i-k|>1\\ %E:fixed
T_iX_iT_i=X_{i+1},\qquad \qquad {T_i}^{-1}Y_i{T_{i}}^{-1}=Y_{i+1}\\
[T_i,X_k]=0,\qquad [T_i,Y_k]=0\ \mbox{for}\ |i-k|>1\\
[X_j,X_k]=0,\qquad [Y_j,Y_k]=0, \qquad Y_1X_1\ldots X_N=qX_1\ldots X_NY_1,$$
$${X_1}^{-1}Y_2=Y_2{X_1}^{-1}{T_1}^{-2} \\
\end{multline} 
\end{definition}

The algebra $\mathbf{H}_N$ contains the Hecke algebra generated by $T_i$, and two copies of the affine Hecke algebra generated by $(T_i,X_j)$ and $(T_i,Y_j)$, respectively. The basic module of $\h_N$ is the polynomial representation:
$$
\h_N \longrightarrow \textrm{End}(\BC[X_1,\ldots, X_N]), %E: fixed
$$
The element $X_i$ acts as multiplication by $x_i$, and $T_i$ acts by the Demazure-Lusztig operator:
$$
T_i=t^{1/2}s_{i}+(t^{1/2}-t^{-1/2}) \frac {s_i-1}{x_i/x_{i+1}-1}
$$
where $s_i=(i , i+1)$ are the simple reflections. Let us define the operators $\partial_i$ on $\BC[X_1,\ldots, X_N])$ by:
$$
\partial_{i}(f)=f(x_1,\ldots,x_{i-1},qx_i,x_{i+1},\ldots x_N).
$$
and introduce the operator $\gamma=s_{N-1}\cdots s_{1}\partial_{1}$ 
Set:
$$
Y_i=t^{\frac {N-1}2} T_{i}\cdots T_{N-1}\gamma T_{1}^{-1}\cdots T_{i-1}^{-1}.
$$
Then the operators  $T_i, X_j, Y_j$ on $V_N$ satisfy the relations of the DAHA. A priori, these operators do not necessarily send the subspace of symmetric polynomials $V_N$ to itself. However, it is well-known that for any symmetric polynomial $f$, we have:
$$
f(X_1,\ldots,X_n):V_N \rightarrow V_N, \qquad f(Y_1,\ldots,Y_n):V_N \rightarrow V_N.
$$ 
We will need the following result of Macdonald:

\begin{proposition}(\cite[eq. VI.3.4]{Macdonald})
On the subspace $V_N$ of symmetric polynomials, the operator $\delta_1:=Y_1+\ldots+Y_N$ can be rewritten as:
\begin{equation}
\label{delta1}
\delta_1=\sum_{i}A_i\partial_i, \qquad \text{where} \quad A_i=\prod_{j\neq i}\frac{tx_i-x_j}{x_i-x_j}. 
\end{equation}
\end{proposition}

In fact, the operator $\delta_1$ is diagonal in the basis of Macdonald polynomials of $V_N$, with eigenvalues given by: 
$$
\delta_1 \cdot P_{\lambda}(x)= \left(\sum_{i}t^{N - i}q^{\lambda_i} \right)P_{\lambda}(x). 
$$
This is part of an alternate definition of Macdonald polynomials as common eigenfunctions of a collection of commuting differential operators, the first of which is $\delta_1$. Generalizing this, the following result is the main theorem of \cite{Ch2}:

\begin{theorem}[\cite{Ch2}]
Let $f$ be a symmetric polynomial in $N$ variables. The operator: 
$$L_{f}:=f(Y_1,\ldots,Y_N)$$ 
is diagonal in the basis of Macdonald polynomials in $V_N$, with eigenvalues:
\begin{equation}
\label{eigenvalues}
L_{f} \cdot P_{\lambda}(x)=f(t^{\rho_N}q^{\lambda})P_{\lambda}(x).
\end{equation}
\end{theorem}

\subsection{The $\slz$ action}
\label{sub:slzact}

Letting $e \in \mathbf{H}_N$ denote the complete idempotent, the spherical DAHA is defined as the subalgebra:
$$
\sh_N = e\cdot \mathbf{H}_N \cdot e
$$ 
There is an action of $\slz$ on $\mathbf{H}_N$ that preserves the subalgebra $\sh_N$. To define it, let:
$$
\tau_{+}=\left(\begin{matrix}1 & 1\\ 0& 1 \end{matrix}\right),\quad \tau_{-}=\left(\begin{matrix}1 & 0\\ 1& 1 \end{matrix}\right)
$$
be the generators of $\slz$. Then Cherednik (\cite{ChBook}, see also \cite{svmacd}) shows that:
\begin{multline}
\label{SL2 action}
\tau_{+}(X_i)=X_i,\quad \tau_{+}(T_i)=T_i,\quad \tau_{+}(Y_i)=Y_iX_i(T_{i-1}^{-1}\cdots T_i^{-1})(T_{i}^{-1}\cdots T_{i-1}^{-1}) \\
\tau_{-}(Y_i)=Y_i,\quad \tau_{-}(T_i)=T_i,\quad \tau_{-}(X_i)=X_iY_i(T_{i-1}\cdots T_i)(T_{i}\cdots T_{i-1})
\end{multline}
extend to automorphisms of $\mathbf{H}_N$, and they respect the relations in $\slz$. This action allows one to construct certain interesting elements in $\mathbf{H}_N$. Start by defining:
$$
P^{\lambda,N} := P_\lambda(Y_1,\ldots,Y_N) \in \sh_N
$$
For example, $P^{(1),N} = \delta_1$ is the sum of the $Y_i$. For any pair of integers $(n,m)$ with $\gcd(n,m)=1$, let us choose any matrix of the form:
$$
\gamma_{n,m} = \left(\begin{matrix} x & n \\ y & m \\ \end{matrix}\right) \in \slz
$$ 

\begin{definition} (and \textbf{Proposition}, see \cite{ChBook}, \cite{svmacd}) The elements:
$$
P_{n,m}^{\lambda,N} := \gamma_{n,m}(P^{\lambda,N}) \in \sh_N
$$
do not depend on the choice of $\gamma_{n,m}$.

\end{definition}

The same construction can be done starting from the power-sum symmetric functions: 
$$
p_k^N:=Y_1^k+\ldots+Y_N^k \in \sh_N.
$$ 
Acting on them with $\gamma_{n,m}$ gives rise to operators:
$$
P^N_{kn,km} := \gamma_{n,m}(p_k^N) \in \sh_N
$$
As in the above Proposition, these do not depend on the particular choice of $\gamma_{n,m}$. Since the Macdonald polynomials $P^{\lambda,N}$ are polynomials in the power-sum functions $p_k^N$, then the $P_{n,m}^{\lambda,N}$ are polynomials in $P^N_{kn,km}$. These can be easily calculated on a computer.

\subsection{Cherednik's conjectures}

Following \cite{Ch}, we define the \emph{DAHA-superpolynomials}: 
$$
\mathcal{P}_{n,m}^{\lambda,N}(q,t) =\varepsilon_N(P_{n,m}^{\lambda,N}\cdot 1) \in \BK, \qquad \qquad \mathcal{P}_{kn,km}^{N}(q,t) =\varepsilon_N(P_{kn,km}^{N}\cdot 1) \in \BK
$$
where $\e_N$ is the evaluation map of \eqref{eqn:eval}. Since the operators $P_{n,m}^{\lambda,N}$ can be expressed as sums of products of $P_{kn,km}^{N}$, we will see that the superpolynomials $\mathcal{P}_{n,m}^{\lambda,N}(q,t)$ can be expressed in terms of the matrix elements of $P_{kn,km}^{N}$. Therefore, we will mostly focus on the latter, and in Section \ref{sec:hilb} we will show how to compute them.  The following conjectures were stated in \cite{Ch}: 

\begin{conjecture}[Stabilization]
\label{conj1}

There exists a polynomial $\mathcal{P}_{n,m}^{\lambda}(u,q,t)$ such that:
$$
\mathcal{P}_{n,m}^{\lambda,N}(q,t) = \mathcal{P}_{n,m}^{\lambda}(u=t^N,q,t)
$$
\end{conjecture}

\begin{conjecture}[Duality]
\label{conj2}
Define the reduced superpolynomial by the equation:
$$\CP_{n,m}^{\lambda,red}(u,q,t)=\frac{\CP_{n,m}^{\lambda}(u,q,t)}{\CP_{1,0}^{\lambda}(u,q,t)}$$
Then the polynomials $\CP$ for transposed diagrams are related by the equation:
$$
q^{(1-n)|\lambda|} \CP_{n,m}^{\lambda^t,red}(u,q,t) = t^{(n-1)|\lambda|} \CP_{n,m}^{\lambda,red}(u,t^{-1},q^{-1}).
$$
\end{conjecture}
One of our main results is a proof of the above conjectures, to be given in Subsections \ref{sub:proofs} and \ref{sub:refined}. 

\subsection{Refined Chern-Simons theory}

In this section we describe the approach of Aganagic and Shakirov, which has its roots in Chern-Simons theory. A 3-dimensional topological quantum field theory associates a number $Z(M)$ to every 3-manifold, a Hilbert space $Z(N)$ to a closed 2-manifold $N$, and a vector $Z(M)\in Z(\partial{M})$ to a 3-manifold $M$ with boundary. All our 3-manifolds may come with closed knots embedded in them.

We will be interested in invariants of torus knots inside the sphere $S^3$. The sphere can be split into two solid tori glued along the boundary, and the Hilbert space $Z(T^2)$ associated to their common boundary can be identified with a suitable quotient of $V$. This quotient has a basis consisting of Macdonald polynomials labeled by Young diagrams inscribed in a $k\times N$ rectangle. One can think of Macdonald polynomials as invariants of the meridian of the solid torus colored by these diagrams. The space $Z(T^2)$ is acted on by the mapping class group of the torus, namely $\slz$, as constructed in \cite{AS}.

Consider the $(m,n)$ torus knot colored by a partition $\lambda$ inside a solid torus, linked with a meridian colored by $\mu$. For such a link, a TQFT should produce a vector $v_{n,m,\mu}^{\lambda}$ in $Z(T^2)$. One defines a {\em knot operator} $W_{n,m}^{\lambda}$ by the formula:
$$
W_{n,m}^{\lambda}|P_{\mu}\rangle=v_{n,m,\mu}^{\lambda}.
$$
In particular, $W_{n,m}^{\lambda}| 1 \rangle$ is the vector in $Z(T^2)$ associated to a solid torus with a $\lambda$-colored $(m,n)$ torus knot inside. The construction of \cite{AS} uses the equation:
\begin{equation}
\label{conj}
W_{n,m}^{\lambda}=K^{-1}W_{1,0}^{\lambda}K,
\end{equation}
where $W_{1,0}^{\lambda}$ is the operator of multiplication by $P_{\lambda}$ and $K$ is any element of $\slz$ taking $(1,0)$ to $(n,m)$, seen as an endomorphism of $Z(T^2)$. The action of $\slz$ was introduced by A. Kirillov Jr. in \cite{K} and is given by the two matrices  (\cite{AS,EK,K}) written in the Macdonald polynomial basis \footnote{The operators in \cite{AS} differ from these by overall scalar factors, which are not important for us}:
\begin{equation}
\label{S and T}
S_{\lambda}^{\mu}=P_{\lambda}(t^{\rho}q^{\mu})P_{\mu}(t^{\rho}),\qquad T_{\lambda}^{\mu}=\delta_{\lambda}^{\mu}q^{\frac{1}{2}\sum_{i}\lambda_i(\lambda_i-1)}t^{\sum_{i}\lambda_i(i-1)}.
\end{equation}
which correspond to the matrices:
$$
\sigma = \left( \begin{array}{cc}
0 & 1 \\
-1 & 0 \end{array} \right), \qquad \tau = \left( \begin{array}{cc}
1 & 1 \\
0 & 1 \end{array} \right) \in \slz.
$$
 The refined Chern-Simons knot invariant is then defined as
\begin{equation}
\label{CS evaluation}
\CP^{CS}_{n,m,\lambda}(q,t):=\langle 1|SW_{n,m}^{\lambda}|1 \rangle_{q,t}.
\end{equation}
The extra $S$-matrix in \eqref{CS evaluation} is responsible for the gluing of two solid tori into $S^3$.
The equations (\ref{conj})-(\ref{CS evaluation}) give a rigorous definition of the polynomials $\CP^{CS}_{n,m\lambda}(q,t)$, which a priori depend on the choice of $k$ and $N$. It is conjectured in \cite{AS} that for large enough $k$ and $N$ the answer does not depend on $k$ and its $N$-dependence can be captured in a variable $u=t^N$. 

The approach of \cite{Ch} is to identify $Z(T^2)$ with the finite-dimensional  representation of $\sh_N$. Such representations were classified in \cite{VV}, and in type $A_N$ they occur when $q$ is a root of unity of degree $k+\beta N$, and $t=q^{\beta}$. It is known (\cite{Ch3}) that the bilinear form $\langle\cdot, \cdot \rangle_{q,t}$ is nondegenerate on $Z(T^2)$, which is a quotient of the polynomial representation by the kernel of this form. 
The following lemma shows the equivalence of the approaches of \cite{AS} and \cite{Ch}.

\begin{lemma}
\label{lem:conj}
In the representation $Z(T^2)$, any $W\in \sh_N$ satisfies:
\begin{equation}
\label{eqn:lemma}
SWS^{-1}=\sigma(W).
\end{equation}
where $\sigma$ acts on the DAHA as in Subsection \ref{sub:slzact}.
\end{lemma}

\begin{proof} By \eqref{eigenvalues}, we have $S_{\lambda}^{\mu}=\e_N(P_{\lambda}(Y)\cdot P_{\mu}(x))$. Consider the vector:
$$
v=S(1)=\sum_{\mu}\e_N(P_{\mu})P_{\mu}
$$
Then:
$$
S(P_{\lambda})=\sum_{\mu}S_{\lambda}^{\mu}P_{\mu}=P_{\lambda}(Y)\cdot v
$$
hence we conclude that for any function $f \in V_N$, one has $S(f)=\sigma(f)\cdot v$. Therefore:
$$
S(Wf)=\sigma(Wf)\cdot v=\sigma(W)\sigma(f)\cdot v=\sigma(W)S(f)
$$
holds for any $W\in \mathbf{H}_N$ and any $f\in V_N$.
\end{proof}

\begin{corollary}
The definitions of the refined knot invariants in \cite{AS} and \cite{Ch}
are equivalent to each other:
$$\CP^{CS}_{n,m,\lambda}(q,t)=\CP^{\lambda,N}_{n,m}(q,t).$$
\end{corollary}

\begin{proof}
Similarly to Lemma \ref{lem:conj}, one can show that the equation $TWT^{-1}=\tau(W)$ holds for any $W\in \sh_N$. Together with (\ref{eqn:lemma}), this implies the equation $KWK^{-1}=\kappa(W)$, where $K$ is an arbitrary operator from $\slz$ and $\kappa$ is the corresponding automorphism of the spherical DAHA. In other words, the two actions of $\slz$ on the image of $\sh_N$ in the automorphisms of $Z(T^2)$ agree with each other. Since both $P_{1,0}^{\lambda,N}$ and $W_{1,0}^{\lambda,N}$ are defined as multiplication operators by the Macdonald polynomial $P_{\lambda}$, one has $P_{1,0}^{\lambda,N}=W_{1,0}^{\lambda,N}$ and
$$P_{n,m}^{\lambda,N}=W_{n,m}^{\lambda,N}$$
for all $m$, $n$ and $\lambda$. It remains to notice that the covector 
$\langle 1|S|\cdot\rangle_{q,t}$ coincides with the evaluation map $\varepsilon_N$, hence
$$\CP^{CS}_{n,m\lambda}(q,t)=\langle 1|SW_{n,m}^{\lambda}|1 \rangle_{q,t}=
\varepsilon_{N}(W_{n,m}^{\lambda}\cdot 1)=\varepsilon_{N}(P_{n,m}^{\lambda}\cdot 1)=\CP^{\lambda,N}_{m,n}(q,t).
%\qedhere
$$
\end{proof}

\section{Stabilization and $N$-dependence}
\label{sec:stab}

Conjecture \ref{conj1} involves the $N$-dependence of the expression $\CP_{n,m}^{\lambda,N} = \varepsilon_N (P^{\lambda,N}_{n,m} \cdot 1)$, and we will understand this in three steps. First, we will recast the evaluation $\e_N$ as a certain matrix coefficient of the operator $P^{\lambda,N}_{n,m}$. Secondly, we will describe the behaviour of these operators  as $N \rightarrow \infty$ and show that they stabilize to an operator $P^\lambda_{n,m}$. Thirdly, we will discuss the behaviour of the matrix coefficients as $N \rightarrow \infty$.

\subsection{Matrix Coefficients}

Let us start from the polynomial representation $\sh_N \longrightarrow \text{End}(V_N)$. We will consider the \emph{evaluation vector}:
\begin{equation}
\label{eqn:evalv}
\Vv(u) := \sum_{\lambda} \frac {J_\lambda }{h_\la h'_\la}  \prod_{\square\in \lambda}\left(t^{l'(\square)}-u q^{a'(\square)}\right)\in V_N
\end{equation}
In the above, the first sum goes over all partitions, and so $\Vv(u)$ rightfully takes values in a completion of $V_N$. Then for any $f\in V_N$, we have:
$$
\e_N(f) = \langle f, \Vv(u) \rangle_{q,t}  \Big |_{u=t^N}
$$
Indeed, since the above relation is linear, it is enough to check it for $f=J_\lambda$, where it follows from \eqref{eqn:normj} and \eqref{J eval}. Then our knot invariants are given by the formula:
\begin{equation}
\label{eqn:pose}
\CP^{\lambda,N}_{n,m} = \e_N(P^{\lambda,N}_{n,m}\cdot 1) = \langle \Vv(u)|P^{\lambda,N}_{n,m}|1\rangle_{q,t} \Big |_{u=t^N}
\end{equation}
where $\langle \cdot | * | \cdot \rangle_{q,t}$ denotes matrix coefficients with respect to the Macdonald scalar product $\langle \cdot , \cdot \rangle_{q,t}$.

\subsection{Stabilization of operators}

We will now let $N$ vary. Recall that the spaces of symmetric polynomials $V_N$ form a projective system under the maps $\eta_{N}:V_N \longrightarrow V_{N-1}$ that set $x_N=0$, and $V$ is the inverse limit if this system. Therefore, we have projection maps $\eta_{\infty,N}: V \longrightarrow V_N$. We will rescale our operators: 
\begin{equation}
\label{eqn:rescale}
\overline{P}^{N}_{0,k}=t^{-k(N-1)}\left(P^{N}_{0,k} - \frac {t^{kN} - 1}{t^k-1} \right), \qquad \qquad \overline{P}^{N}_{kn,km} = P^N_{kn,km} \quad \text{for }n \neq 0
\end{equation}

%\begin{equation}
%\label{eqn:rescale}
%\overline{P}^{N}_{kn,km}=t^{\frac{km(N-1)}{2}}P^{N}_{kn,km} \qquad \mbox{for } n\neq 0.
%\end{equation}
 
\begin{proposition}(cf. \cite[Prop. 1.4]{svhilb}) The following relation holds: 

$$
\overline{P}^{N-1}_{kn,km}\circ \eta_{N}=\eta_{N}\circ \overline{P}^{N}_{kn,km}
$$
\end{proposition} 

Therefore, there exist limiting operators $P_{kn,km}:=\lim_{N\to\infty} \overline{P}^{N}_{kn,km}$ on $V$, such that:
$$
\overline{P}^{N}_{kn,km}\circ \eta_{\infty,N}=\eta_{\infty,N}\circ P_{kn,km}
$$
For a general partition, the operators $P_{n,m}^{\lambda,N}$ on $V_N$ are sums of products of $P_{kn,km}^N$, and therefore there exist operators $P_{n,m}^\lambda$ on $V$ which stabilize the operators $P_{n,m}^{\lambda,N}$:
$$
\overline{P}^{\lambda,N}_{n,m} \circ \eta_{\infty,N} = \eta_{\infty,N} \circ P_{n,m}^\lambda
$$
All these new operators $P^\lambda_{n,m}$ and $P_{kn,km}$ lie in the algebra $\sh$, defined in \cite{svhilb} as the stabilization of the spherical DAHA's $\sh_N$ as $N \rightarrow \infty$. 

\subsection{Commutation relations}

The isomorphism between $\sh$ and the elliptic Hall algebra, established in \cite{svmacd} (see also \cite{Negut2}), allows one to present some explicit commutation relations between the operators $P_{n,m}$. These commutation relations were discovered by Burban and Schiffmann in \cite{BS}. It is sometimes convenient to represent the operator $P_{n,m}$ by the vector $(n,m)$ in the integer lattice $\mathbb{Z}^2$ (we assume $n>0$). The action of $\slz$ on the algebra $\sh$ is then just given by the linear action on this lattice. 

\begin{definition}(\cite{Negut2})
A triangle with vertices $X=(0,0)$, $Y=(n_2,m_2)$ and $Z=(n_1+n_2,m_1+m_2)$ is called quasi-empty if $m_1n_2-m_2n_1>0$ and there are no lattice points neither inside the triagle, nor on at least one of the edges $XY, YZ$.
\end{definition}

Let us define the constants:
\begin{equation}
\label{eq:alpha}
\alpha_n=\frac{(q^n-1)(t^{-n}-1)(q^{-n}t^{n}-1)}{n},
\end{equation}
and the operators $\theta_{kn,km}$ (for coprime $m,n$) by the equation: 
$$
\sum_{n=0}^{\infty}z^{n}\theta_{kn,km}=\exp\left(\sum_{n=1}^{\infty}\alpha_nz^{n}P_{kn,km}\right).
$$
The elliptic Hall algebra is defined by the following commutation relations (\cite{BS}):
$$
[P_{n_1,m_1},P_{n_2,m_2}]=0,
$$
if the vectors $(n_1,m_1)$ and $(n_2,m_2)$ are collinear, and:
\begin{equation}
\label{commutation}
[P_{n_1,m_1},P_{n_2,m_2}]=\frac{\theta_{n_1+n_2,m_1+m_2}}{\alpha_1},
\end{equation}
if the points $(0,0), (n_2,m_2)$ and $(n_1+n_2,m_1+m_2)$ form a quasi-empty triangle.

\begin{figure}[ht]
\begin{tikzpicture}
\draw [->,>=stealth] (0,0)--(4,0);
\draw [->,>=stealth] (0,0)--(0,4);
\draw [->,>=stealth] (0,0)--(1,1);
\draw [->,>=stealth] (1,1)--(2,3);
\draw [->,>=stealth] (0,0)--(2,3);
\draw (0,0) node {\tiny $\bullet$};
\draw (1,0) node {\tiny $\bullet$};
\draw (2,0) node {\tiny $\bullet$};
\draw (3,0) node {\tiny $\bullet$};
\draw (0,1) node {\tiny $\bullet$};
\draw (1,1) node {\tiny $\bullet$};
\draw (2,1) node {\tiny $\bullet$};
\draw (3,1) node {\tiny $\bullet$};
\draw (0,2) node {\tiny $\bullet$};
\draw (1,2) node {\tiny $\bullet$};
\draw (2,2) node {\tiny $\bullet$};
\draw (3,2) node {\tiny $\bullet$};
\draw (0,3) node {\tiny $\bullet$};
\draw (1,3) node {\tiny $\bullet$};
\draw (2,3) node {\tiny $\bullet$};
\draw (3,3) node {\tiny $\bullet$};
\end{tikzpicture}
\caption{Example of a commutation relation: $[P_{1,2},P_{1,1}]=P_{2,3}$.}
\end{figure}
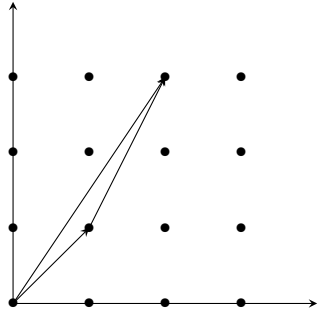

It would be interesting to find a topological interpretation of the equation \eqref{commutation}. Note that the complete statement of this equation requires the RHS of \eqref{commutation} to be multiplied by a central element, which must be included in the definition of the elliptic Hall algebra. For the sake of brevity, we will neglect this technical point, whose discussion can be found in \cite{BS}.

\subsection{Proof of the stabilization conjecture \ref{conj1}}
\label{sub:proofs}

In this Subsection, we will use the elliptic Hall algebra viewpoint to prove Conjecture \ref{conj1}. The inner product $\langle \cdot, \cdot \rangle_{q,t}$, the vacuum vector $|1\rangle$ and the evaluation vector $\Vv(u)$ are all preserved under the maps $\eta_N:V_N \longrightarrow V_{N-1}$. Therefore, all of these notions are compatible with the above limiting procedure, so we conclude that:
\begin{equation}
\label{eqn:superpoly}
\CP^\lambda_{n,m} := \langle \Vv(u)|P^{\lambda}_{n,m}|1\rangle_{q,t} 
\end{equation}
is the same whether we compute it in $V$ or $V_N$. Comparing this with \eqref{eqn:pose}, we conclude that:
$$
\CP^\lambda_{n,m} = \CP^{\lambda,N}_{n,m} \Big |_{u=t^N}
$$ 
This proves Conjecture \ref{conj1}. We will prove Conjecture \ref{conj2} in the next section, when we will realize the elliptic Hall algebra and the polynomial representation geometrically.

\section{The Hilbert scheme of points}
\label{sec:hilb}

\subsection{Basic Definitions} Let $\Hilb_d = \Hilb_d(\mathbb{C}^2)$ denote the moduli space of colength $d$ ideal sheaves $\CI \subset \CO_{\BC^2}$. It is a smooth and quasi-projective variety of dimension $2d$. Pushing forward the universal quotient sheaf on $\Hilb_d \times \BC^2$ gives rise to the rank $d$ tautological vector bundle on $\Hilb_d$:
$$
\CT|_{\CI} = \Gamma(\BC^2, \CO/\CI)
$$
The torus $T = \BC^* \times \BC^*$ acts on $\BC^2$ by dilations, and therefore acts on sheaves on $\BC^2$ by direct image. This gives an action of $T$ on the moduli spaces $\Hilb_d$. We can then consider the equivariant $K-$theory groups $K_T^*(\Hilb_d)$, all of which will be modules over:
$$
K_T^*(\textrm{pt}) = \BK_0 = \BC[q^{\pm 1}, t^{\pm 1}]
$$
where $q$ and $t$ are a fixed basis of the characters of $T = \BC^* \times \BC^*$. As early as the work of Nakajima, it became apparent that one needs to study the direct sum of these $K-$theory groups over all degrees $d$. In other words, we will consider the vector space:
$$
K = \bigoplus_{d\geq 0} K_T^*(\Hilb_d) \otimes_{\BK_0} \BK
$$
This vector space comes with the geometric pairing:
\begin{equation}
\label{eqn:geom}
(\cdot,\cdot): K \otimes K \longrightarrow \BK, \qquad \qquad (\alpha, \beta) = \pi_*(\alpha \otimes \beta)
\end{equation}
where $\pi:\Hilb_d \longrightarrow \text{pt}$ is the projection map. Feigin-Tsymbaliuk and Schiffmann-Vasserot independently proved the following result:

\begin{theorem} \label{thm:sv} (\cite{FT}, \cite{svhilb}) There exists a geometric action of the algebra $\sh$ on $K$, which becomes isomorphic to the polynomial representation $V$. \end{theorem}

\subsection{Torus fixed points}

We have the following localization theorem in equivariant $K-$theory:
\begin{equation}
\label{eqn:localization}
K \cong \bigoplus_{d\geq 0} K_T^*(\Hilb^T_d) \otimes_{\BK_0} \BK
\end{equation}
There are finitely many torus fixed points in $\Hilb_d$, and they are all indexed by partitions $\lambda \vdash d$:
$$
\CI_\lambda = (x^{\lambda_1} y^0, x^{\lambda_2}y^1,\ldots) \subset \BC[x,y]
$$
The skyscraper sheaves at these fixed points give a basis $I_\lambda = [\CI_\lambda]$ of the right hand side of \eqref{eqn:localization}, and therefore also of the vector space $K$. This basis is orthogonal with respect to the pairing \eqref{eqn:geom}:
$$
(I_\lambda, I_\mu) = \delta_{\lambda,\mu} \cdot g_\lambda
$$
where:
\begin{equation}
\label{eqn:defg}
g_\lambda := \Lambda^\bullet(T^\vee_\lambda \Hilb_d) = \prod_{\square \in \lambda}(1-q^{a(\square)}t^{-l(\square)-1}) \prod_{\square \in \lambda}(1-q^{-a(\square)-1}t^{l(\square)})
\end{equation}

\begin{theorem} (e.g. \cite{haimlectures2,OP}) 
\label{thm:fixed} 
Under the isomorphism $K \cong V$ of Theorem \ref{thm:sv}, the classes $I_\lambda \in K$ correspond to the modified Macdonald polynomials:
$$
\Hl(q;t) = t^{n(\lambda)} \ph \left[J_{\lambda}(q;t^{-1}) \right] \in V
$$
where $\ph:V \longrightarrow V$ is the plethystic homomorphism defined by $p_k \longrightarrow \frac {p_k}{1-1/t^{k}}$, and:
\begin{equation}
\label{eqn:defn}
n(\la)=\sum_{\sq\in \la} l(\sq)=\sum_{\sq\in \la} l'(\sq)
\end{equation}

\end{theorem}
 
The polynomials $\Hl$ were introduced by Garsia and Haiman \cite{GH93}, and they are called {\em modified Macdonald polynomials}. They behave nicely under transpose:
\begin{equation}
\label{hmu symmetry}
\widetilde{H}_{\la^{t}}(q;t)=\widetilde{H}_{\la}(t;q).
\end{equation}
Following \cite{BG}, let us define the operator $\nabla$ on $V$ by the formula:
\begin{equation}
\label{def:nabla}
\nabla \Hl=q^{n(\la^t)}t^{n(\la)}\Hl.
\end{equation}
Under the isomorphism $V\cong K$, $\nabla$ corresponds to the operator of multiplication by the line bundle $\mathcal{O}(1)$, while the geometric inner product \eqref{eqn:geom} corresponds to the following twist of the Macdonald inner product:
\begin{equation}
\label{eqn:twist}
(\cdot, \cdot) = (-q)^{-d} \left\langle \phh \left(\nabla^{-1}(\cdot)\right) , \phh(\cdot) \right\rangle_{q,t^{-1}} \qquad \text{on} \quad K_d \cong V_d 
\end{equation}

\subsection{Geometric operators}
\label{sub:fhilb}
%\subsection{The elements $P_{0,k}$} 
%\label{sub:cartan}

The most interesting elements of $\sh$ to us are $P_{kn,km}$, for all $k\neq 0$ and $\gcd(n,m)=1$. In the current geometric setting, we will study their conjugates:
\begin{equation}
\label{eq:TP0}
\TP_{kn,km} = \ph \circ P_{kn,km}(q;t^{-1}) \circ \phh
\end{equation}

%as $\Hl t^{-n(\lambda)} \prod_{\square \in \lambda} \left(1-q^{a(\square)}t^{-l(\square)-1}\right)^{-1}$ expands in power sum functions $p_k$. 
Let us now describe how these operators act on $K$, and we will start with the simplest case, namely $n=0$. We define the polynomial:
\begin{equation}
\label{eqn:tautclass}
\Lambda(z) = \sum_{i=0}^d (-z)^i [\Lambda^i \CT] \in K[z]
\end{equation}
where $\CT$ is the tautological rank $d$ vector bundle on $\Hilb_d$. As was shown in \cite{Negut}, we have:
\begin{equation}
\label{eqn:0}
\exp\left( \sum_{k \geq 0} \alpha_k \TP_{0,k} z^{k} \right) \cdot c = \frac {\Lambda\left(\frac qz\right)\Lambda\left(\frac tz\right)\Lambda\left(\frac 1{zqt}\right)}{\Lambda\left(\frac 1{zq}\right)\Lambda\left(\frac 1{zt}\right)\Lambda\left(\frac {qt}{z}\right)} \cdot c \qquad \forall c\in K
\end{equation}
where the constants $\alpha_k$ are given by \eqref{eq:alpha}. In order to compute how each $\TP_{0,\pm k}$ acts, we need to expand the right hand side in powers of $z$ and take the appropriate coefficient. We will give a more computationally useful description of $\TP_{0,k}$ in \eqref{eqn:cartan} below. 

%\subsection{The elements $P_{kn,km}$} 
%\label{sub:fhilb}

For $n > 0$, we consider the flag Hilbert scheme:
$$
\Hilb_{d,d+kn} = \{\CI_0 \supset \CI_{1} \supset \ldots \supset \CI_{kn}\} \subset \Hilb_d \times \Hilb_{d+1} \times \ldots \times \Hilb_{d+kn}
$$
where the inclusions are all required to be supported at the same point of $\BC^2$. This variety comes with projection maps: 
$$
p^-: \Hilb_{d,d+kn} \longrightarrow \Hilb_d, \qquad p^+: \Hilb_{d,d+kn} \longrightarrow \Hilb_{d+kn}
$$
that forget all but the first/last ideal in the flag, and with tautological line bundles $\CL_1,\ldots,\CL_{kn}$ given by:
$$
\CL_i|_{\CI_0 \supset \ldots \supset \CI_{kn}} = \Gamma(\BC^2, \CI_{i-1}/\CI_{i})
$$ 
As explained in \cite{Negut}, the flag Hilbert scheme is not simply the iteration of $kn$ individual Nakajima correspondences. The reason for this is that the convolution product of Nakajima correspondences is not a complete intersection, and thus the intersection-theoretic composition is a codimension $kn-1$ class on the scheme $\Hilb_{d,d+kn}$. For this reason, the composition differs significantly from the fundamental class of $\Hilb_{d,d+kn}$. In {\em loc. cit.}, we have the following description of the operators $\TP_{kn,km}$ acting on $K$:
\begin{equation}
\label{eqn:cre}
\TP_{kn,km}\cdot c = \frac 1{[k]} \cdot p^{+}_* \left[\prod_{i=1}^{kn} [\CL_{kn+1-i}]^{\otimes S_{m/n}(i)} \otimes \left( \sum_{j=0}^{k-1}  \frac {(qt)^j \cdot [\CL_n][\CL_{2n}]\cdots [\CL_{jn}]}{[\CL_{n+1}][\CL_{2n+1}]\cdots [\CL_{jn+1}]} \right) \cdot p^{- *}(c)\right]
\end{equation}
where we denote $[k] = \frac {(q^k-1)(t^k-1)}{(q-1)(t-1)}$ and $S_{m/n}$ are the integral parts of \eqref{def:smn}. Note that the above differs by an overall normalizing constant from the operators of \cite{Negut}. The above formula has geometric meaning, and in the next section we will make it more suitable for computations. 

%In the meantime, we will use it to prove Proposition \ref{prop:conj}: 

%\begin{proof} \emph{of Proposition \ref{prop:conj}:} By using \eqref{eqn:r}, the desired property comes down to:
%$$
%\omega_{q,t} \circ P_{kn,km}(q,t) \circ \omega_{q,t}^{-1} = P_{kn,km}(t^{-1},q^{-1}) \cdot (-1)^{k-1} \frac {1-q^k}{1-t^{k}}
%$$ 
%for any $k>0$. We may rewrite this in terms of the modified operators \eqref{eq:TP0} as:
%$$
%\omega_{q,t^{-1}} \circ \phh \circ \TP_{kn,km} \circ \ph \circ \omega_{q,t^{-1}}^{-1} = \phhq \circ \TP_{kn,km}|_{q\leftrightarrow t} \circ \phq \cdot (-1)^{k-1} \frac {1-q^k}{1-t^{-k}}
%$$ 
%Since $\omega_{q,t^{-1}} \circ \phh = \phhq$, the above comes down to the fact that the operator $\TP_{kn,km}(1-t^k)$ is invariant under switching $q$ and $t$, up to a factor of $(-1)^{k-1}$. This follows from the $q,t-$symmetry of the Hilbert scheme, as well as of the line bundles and the flag Hilbert scheme that appear in \eqref{eqn:cre}. 

%\end{proof}

%\begin{equation}
%\label{eqn:ann}
%\TP_{-kn,km}\cdot c = \frac {(-1)^{k-1}}{\beta_k} \cdot p^{-}_* \left[\prod_{i=1}^{kn} [\CL_{kn+1-i}]^{\otimes S_{m/n}(i)-1} \left( \sum_{j=0}^{k-1}  %\frac {(qt)^j  \cdot [\CL_n]...[\CL_{jn}]}{[\CL_{n+1}]...[\CL_{jn+1}]} \right) \cdot p^{+ *}(c)\right]
%\end{equation}

\subsection{Matrix Coefficients} Let us compute the matrix coefficients $(I_\lambda|\TP_{kn,km}|I_\mu )$. Define:
\begin{equation}
\label{eqn:defomega}
\omega(x) = \frac {(x-1)(x-qt)}{(x-q)(x-t)}
\end{equation}
The easiest case for us is $n=0$, since $\TP_{0,k}$ is multiplication with a given $K-$theory class, and thus is diagonal in the basis $I_\lambda$. More concretely, \eqref{eqn:0} yields:
$$
\exp\left( \sum_{k \geq 0} \alpha_k ( I_\lambda|\TP_{0,k}|I_\mu ) z^{k} \right) = \delta_{\lambda}^{\mu}  \prod_{\square \in \lambda}  \frac {\omega\left( \frac z{\chi(\square)} \right)}{\omega \left( \frac {\chi(\square)}z\right)} 
%\frac {(1-zq^{-1}\chi_{\square})(1-zt^{-1}\chi_{\square})(1-zqt\chi_{\square})}{(1-zq\chi_{\square})(1-zt\chi_{\square})(1-zq^{-1}t^{-1}\chi_{\square})}
$$
where given a box $\square = (i,j)$ in a Young diagram, its \emph{weight} is $\chi(\square) = q^{i-1}t^{j-1}$. Taking the logarithm of the above gives us:
\begin{equation}
\label{eqn:cartan}
( I_\lambda|\TP_{0,k}|I_\mu ) = \delta_{\lambda}^{\mu} \left( \frac {(qt)^k}{(q^k-1)(t^k-1)} + \sum_{\square \in \lambda} \chi(\square)^{k} \right)
\end{equation}
The matrix coefficients of $\TP_{kn,km}$ for $n>0$ are written in terms of \textit{standard Young tableaux} (abbreviated $\syt$), so let us recall this notion. Given two Young diagrams $\rho_1 \supset \rho_2$, a $\syt$ between them is a way to index the boxes of $\rho_1 \backslash \rho_2$ with different numbers $1,\ldots,l$ such that any two numbers on the same row or column decrease as we go to the right or up. We will often write:
$$
\rho_1 = \rho_2 + \square_1 + \ldots + \square_l
$$ 
if we want to point out that the box indexed by $i$ is $\square_i$. Then \cite{Negut} gives us the formula:
$$
( I_\mu|\TP_{kn,km}|I_\lambda ) = \frac {\gamma^{kn}}{[k]} \cdot \frac {g_\lambda}{g_\mu} \sum^{\syt}_{\mu = \lambda + \square_1 + \ldots + \square_{kn}} \quad \left[\sum_{j=0}^{k-1}  (qt)^{j} \frac {\chi_{n(k-1)+1}\chi_{n(k-2)+1}\cdots \chi_{n(k-j)+1}}{\chi_{n(k-1)}\chi_{n(k-2)}\cdots \chi_{n(k-j)}} \right] \cdot 
$$
\begin{equation}
\label{eqn:fix3}
\cdot \frac {\prod_{i=1}^{kn} \chi_i^{S_{m/n}(i)} (qt \chi_i - 1 ) }{\left(1 - qt\frac {\chi_2}{\chi_1}\right)\cdots \left(1 - qt\frac {\chi_{kn}}{\chi_{kn-1}} \right)} \prod_{1\leq i<j \leq kn} \omega^{-1} \left(\frac {\chi_j}{\chi_i} \right)  \prod_{1\leq i \leq kn}^{\square \in \lambda} \omega^{-1} \left(\frac {\chi(\square)}{\chi_i} \right) 
\end{equation}
where $\chi_i=\chi(\square_i)$, $\gamma = \frac {(q-1)(t-1)}{qt(qt-1)}$ and $g_\lambda$ are the equivariant constants of \eqref{eqn:defg}.

\begin{remark}
Since the multiplication operators by power sums $p_n$ coincide with the operators $P_{n,0}$, equation \eqref{eqn:fix3} can be used to compute their matrix elements in the modified Macdonald basis. Indeed, this computation will agree with the Pieri rules for Macdonald polynomials \cite{Macdonald}, so \eqref{eqn:fix3} can be considered as a generalization of Pieri rules.
\end{remark}

%$$
%( I_\mu|\TP_{-kn,km}|I_\lambda ) = -\frac 1{\beta_k} \cdot \frac {g_\lambda}{g_\mu} \sum^{\syt}_{\lambda = \mu + \square_1 + ... + \square_{kn}} \quad \left[\sum_{j=0}^{k-1}  (qt)^{j} \frac {\chi_{n(k-1)+1}...\chi_{n(k-j)+1}}{\chi_{n(k-1)}...\chi_{n(k-j)}} \right] \cdot
%$$
%\begin{equation}
%\label{eqn:fix2}
%\cdot \frac {\prod_{i=1}^{kn} \chi_i^{S_{m/n}(i)}(\chi_i - 1)^{-1} }{\left(1 - \frac {\chi_2qt}{\chi_1} \right)...\left(1-\frac {\chi_{kn}qt}{\chi_{kn-1}} \right)} \prod_{1\leq i<j \leq kn} \omega \left(\frac {\chi_i}{\chi_j} \right)  \prod_{1\leq i \leq kn}^{\square \in \mu} \omega \left(\frac {\chi_i}{\chi(\square)} \right)
%\end{equation}

%\begin{remark}
%\label{rem:zeroes}

%The reader will notice that each summand in \eqref{eqn:fix3} above has precisely $kn$ factors equal to 0 in the denominator, so we need to remove them. Rigorously, what one needs to do is to replace each factor $(1-q^at^b)$ in the second line of \eqref{eqn:fix3} by $(1-eq^at^b)$. Then one needs to multiply the right hand side of \eqref{eqn:fix3} by $(1-e)^{kn}$, and let $e\to 1$. This gives the correct meaning of expression \eqref{eqn:fix3}. 

%\end{remark}

%Changing the variables, we see that:
%$$
%\CP^\lambda_{n,m}(u,q,t^{-1}) = \langle \widetilde{\Vv}(u)|P^{\lambda}_{n,m}|_{t\rightarrow t^{-1}} |1\rangle_{q,t^{-1}},\ \widetilde{\Vv}(u) = \sum_{\mu \vdash n} \frac {J_\mu(q;t^{-1}) }{h_\mu(q;t^{-1}) h'_\mu(q;t^{-1})}  \prod_{\square\in \mu}\left(t^{-l'(\square)}-uq^{a'(\square)}\right) 
%$$

\subsection{Refined invariants via Hilbert schemes}
\label{sub:refined}

Similarly to \eqref{eq:TP0}, we can define operators $\TP_{n,m}^{\lambda}$ as conjugates to $P_{n,m}^{\lambda}$ under $\ph$.
In particular, the operator $\TP_{1,0}^{\lambda}$ is conjugate by  $\ph$ to the multiplication operator by $P_{\lambda}(q,t^{-1})$.
Since $\ph$ is a ring homomorphism,  $\TP_{1,0}^{\lambda}$ is a multiplication operator by
$$
\ph(P_{\lambda}(q,t^{-1}))=\frac{t^{-n(\lambda)}\Hl(q,t)}{h_{\lambda}(q,t^{-1})}.
$$
and the action of $\slz$ implies:
\begin{equation}
\label{eq:TP}
\TP_{n,m}^{\la}:= \frac{t^{-n(\lambda)}}{h_{\lambda}(q,t^{-1})}\Hl\left[p_k\to \TP_{kn,km}\right].
\end{equation}
As was shown in \eqref{eqn:superpoly}, the super-polynomials are given by:
$$
\CP^\lambda_{n,m}(u,q,t) = \left\langle \Vv(u)|P^{\la}_{n,m}|1\right \rangle_{q,t},\ \text{where}\ \Vv(u) = \sum_{\mu \vdash n} \frac {J_\mu }{h_\mu(q;t) h'_\mu(q;t)}  \prod_{\square\in \mu}\left(t^{l'(\square)}-uq^{a'(\square)}\right) 
$$
Under the isomorphism of Theorem \ref{thm:fixed}, we can write the above as a matrix coefficient in $K \cong V$:
\begin{equation}
\label{eqn:super2}
\TCP^\lambda_{n,m}(u,q,t) =\left ( \Lambda(u) |\TP^{\lambda}_{n,m}|1 \right) ,\ \text{where}\ \Lambda(u) = \sum_{\mu \vdash n} \frac {I_\mu}{g_\mu}  \prod_{\square\in \mu}\left(1 - u q^{a'(\square)}t^{l'(\square)}\right) 
\end{equation}
and the change of variables is:
\begin{equation}
\label{eqn:change}
\TCP^\lambda_{n,m}(u,q,t) = (-q)^{-n|\lambda|}\CP^\lambda_{n,m}(u,q,t^{-1})
\end{equation}
By the equivariant localization formula, we see that the $K-$theory class $\Lambda(u)$ defined above as a sum of fixed points coincides with the exterior class of \eqref{eqn:tautclass}. We have thus expressed our super-polynomials in terms of the geometric operators \eqref{eqn:cre} on the $K-$theory of the Hilbert scheme. 

\begin{proof}[Proof of Theorem \ref{thm:main}] By \eqref{eqn:super2}, the uncolored DAHA -- superpolynomial is given by:
$$
\TCP_{n,m}(u,q,t) = ( \Lambda(u) | \TP_{n,m} | 1 ) = \sum_{\mu \vdash n} ( I_\mu | \TP_{n,m} | 1 ) \cdot \frac {\prod_{\square \in \mu} (1-u\chi(\square))}{g_\mu}
$$
We can use \eqref{eqn:fix3} to compute the above matrix coefficients, and we obtain:
\begin{equation}
\label{eqn:sumsyt}
\TCP_{n,m}(u,q,t) =  \sum^{\syt}_{\mu = \square_1+\ldots+\square_n} \frac {\gamma^n}{g_\mu} \cdot \frac {\prod_{i=1}^{n} \chi_i^{S_{m/n}(i)} (1-u\chi_i)(qt\chi_i - 1)}{\left(1 - qt \frac {\chi_{2}}{\chi_{1}}\right) \cdots  \left(1 - qt \frac {\chi_{n}}{\chi_{n-1}}\right)}\prod_{1\leq i < j\leq n} \omega^{-1} \left( \frac {\chi_j}{\chi_i} \right) 
\end{equation}
Changing $t\rightarrow t^{-1}$ gives us formula \eqref{eqn:for}.
\end{proof}

%\subsection{The colored case} 
%\label{sub:color}
As for the colored knot invariant $\TCP_{n,m}^\lambda$ %of Conjecture \ref{conj1}
, it is also a particular matrix coefficient of the operator $\TP_{n,m}^\lambda$:
$$
\TCP^\lambda_{n,m}(u,q,t) = ( \Lambda(u) | \TP^\lambda_{n,m} | 1 ) = \sum_{\mu \vdash n|\lambda|} ( I_\mu | \TP^\lambda_{n,m} | 1 ) \cdot \frac {\prod_{\square \in \mu} (1 - u\chi(\square))}{g_\mu}
$$
By \eqref{eq:TP}, the operator $\TP_{n,m}^\lambda$ expands in the operators $\TP_{nk,mk}$ given by the same formula as modified Macdonald polynomials expand in the power-sum functions $p_k$. For computational purposes, the task then becomes to compute expressions of the form:
$$
\sum_{\mu \vdash n(k_1+\ldots+k_t)} ( I_\mu | \TP_{nk_1,mk_1}\cdots\TP_{nk_t,mk_t} | 1 ) \cdot \frac {\prod_{\square \in \mu} (1 - u\chi(\square))}{g_\mu}
$$
for any $k_1,\ldots,k_t$. These can be computed by iterating \eqref{eqn:fix3} and the result will be a sum over standard Young tableaux. The summand will be in general more complicated than \eqref{eqn:sumsyt}, but it can be taken care of by a computer. While we do not yet have a ``nice" formula suitable for writing down in a theoretical paper, we may use the geometric viewpoint to prove Cherednik's second conjecture:

\begin{proof}[Proof of Conjecture \ref{conj2}] In terms of the modified superpolynomials $\TCP$, the desired identity becomes:
$$
\TCP_{n,m}^{\lambda^t,red}(u,q,t) = \TCP_{n,m}^{\lambda,red}(u,t,q),\ \text{where}\ \TCP_{n,m}^{\lambda,red}(u,q,t):=(-q)^{-(n-1)|\lambda|}\CP_{n,m}^{\lambda,red}(u,q,t^{-1}).
$$
%In other words, we need to show that the reduced superpolynomial is invariant under switching $q\leftrightarrow t$ and $\lambda \leftrightarrow \lambda^t%$, up to a power of $q$.  
 Recall that 
$$
\CP_{1,0}^{\lambda}(u,q,t)=\frac{1}{h_{\lambda}(q,t)}\prod_{\sq\in \lambda}(t^{l'(\sq)}-uq^{a'(\sq)}),
$$
so
$$
\TCP_{1,0}^{\lambda}(u,q,t)=\frac{t^{-n(\lambda)}}{h_{\lambda}(q,t^{-1})}\prod_{\sq\in \lambda}(1-uq^{a'(\sq)}t^{l'(\sq)}),
$$
hence
$$
\TCP_{n,m}^{\lambda,red}(u,q,t)=\frac{\TCP_{n,m}^{\lambda}(u,q,t)}{\TCP_{1,0}^{\lambda}(u,q,t)}=
\frac{h_{\lambda}(q,t^{-1})(\Lambda(u) |\TP^{\lambda}_{n,m}|1 )}{t^{-n(\lambda)}\prod_{\sq\in \lambda}(1-uq^{a'(\sq)}t^{l'(\sq)})}.
$$
Therefore by \eqref{eq:TP} we get
$$
\TCP_{n,m}^{\lambda,red}(u,q,t)=\frac{\left(\Lambda(u) |\Hl(q,t)\left[p_k\to \TP_{km,kn}\right]|1\right)}{\prod_{\sq\in \lambda}(1-uq^{a'(\sq)}t^{l'(\sq)})}.
$$
By \eqref{hmu symmetry} $\Hl(q,t)=\Hlt(t,q)$, and by \eqref{eqn:fix3} the matrix coefficients of $\TP_{km,kn}$ are invariant under the 
switching  $q\leftrightarrow t$ and the transposition of the axis, as well as $\Lambda(u)$. 
\end{proof}

\subsection{Constant term formulas} Formula \eqref{eqn:sumsyt} can be repackaged as a contour integral. We may write:
$$
\TCP_{n,m}(u,q,t) = ( \Lambda(u) |\TP_{n,m}|1 )  = (1|\TP_{-n,m}|\Lambda(u) )
$$
where $\TP_{-n,m}$ is the adjoint of $\TP_{n,m}$. Formula (4.12) of \cite{Negut} gives us the following integral formula for this expression: 
\begin{equation}
\label{eqn:integral}
\TCP_{n,m}(u,q,t) = \int
\frac {\prod_{i=1}^n z_i^{S_{m/n}(i)} \cdot \frac {1 - u z_i}{z_i-1} }{\left(1-qt\frac {z_{2}}{z_{1}}\right) \cdots  \left(1-qt\frac {z_{n}}{z_{n-1}}\right)} \prod_{1\leq i < j\leq n} \omega \left( \frac {z_i}{z_j} \right)  \frac {dz_1}{2\pi i z_1}\cdots \frac {dz_n}{2\pi i z_n}  
\end{equation}
where the contours of the variables $z_i$ surround 1, with $z_1$ being the outermost and $z_{n}$ being the innermost (we take $q,t$ very close to 1). We can move the contours so that they surround $0$ and $\infty$, and then the integral comes down to the following residue computation:
$$
\TCP_{n,m}(u,q,t) = \left(\res_{z_{n}=0} - \res_{z_{n}=\infty} \right) \cdots \left( \res_{z_1=0} - \res_{z_1=\infty}\right) 
$$
\begin{equation}
\label{eqn:resknot}
\frac 1{z_1\cdots z_n} \cdot \frac {\prod_{i=1}^n z_i^{S_{m/n}(i)} \cdot \frac {1-uz_i}{z_i-1} }{\left(1 - qt\frac {z_{2}}{z_{1}}\right) \cdots \left(1 - qt\frac {z_{n}}{z_{n-1}}\right)} \prod_{1\leq i < j\leq n} \omega \left( \frac {z_i}{z_j} \right) 
\end{equation}
We will compute the above residue in section \ref{sub:tesler}, which will give another combinatorial way to compute $\TCP_{n,m}$. 

\subsection{} Let us say a few words about the viewpoint of Nakajima in \cite{Nak}, which relates knot invariants to the following map:
$$
\Psi_d:V \longrightarrow K_d, \qquad \qquad \Psi_d(f) = f(\CW^\vee)
$$
where $\CW$ is the universal bundle on $\Hilb_d$. As a $K-$theory class (and this will be sufficient for the purposes of the present paper), it is given by $[\CW] = 1 - (1-q)(1-t)[\CT]$. We may take the direct product of the above maps over all $d$ and define:
\begin{equation}
\label{eqn:s}
\Psi:V \longrightarrow K, \qquad \qquad \Psi = \prod_{d=0}^\infty \Psi_d
\end{equation} 
The map $\Psi$ defined above takes values in a certain completion of $K$, since we consider the direct product. Via equivariant localization, we see that:
\begin{equation}
\label{eqn:sloc}
\Psi(f) = \sum_{\lambda} \frac {I_\lambda}{g_\lambda} \cdot f\left( 1 - (1-q^{-1})(1-t^{-1})\sum_{\square \in \lambda} \chi(\square)^{-1} \right) 
\end{equation}
The right hand side of \eqref{eqn:sloc} uses plethystic notation of symmetric functions, which is described in \cite{Nak}. In \emph{loc. cit.}, Nakajima uses the map \eqref{eqn:s} to study knot invariants, essentially by using the viewpoint given by the left hand side of relation \eqref{eqn:lemma} where the $S$-matrix is realized as an operator on $V_N$. Our viewpoint, outlined in the previous sections, is to compute the same knot invariants by using the right hand side of \eqref{eqn:lemma} and interpret $S$ as an automorphism of the algebra $\sh$. The two perspectives produce significantly different formulas.

%Then we define $e_n(\CW^\vee) = [\Lambda^n \CW^\vee] \in K$ and extend the map \eqref{eqn:s} multiplicatively: $\Psi$ applied to the usual product of symmetric functions on $V$ corresponds to the tensor product of $K-$theory classes. 

\section{Representations of the rational Cherednik algebra}
\label{sec:rat}

\subsection{The rational Cherednik algebra}

Rational Cherednik algebras were introduced in \cite{EG} as degenerations of the DAHA.

\begin{definition}
The rational Cherednik algebra of type $A_{n-1}$ with parameter $c$ is:
$$
\mathbf{H}_{c}=\mathbb{C}[\mathfrak{h}]\otimes \mathbb{C}[\mathfrak{h}^{*}]\rtimes \mathbb{C}[S_n],
$$
where $\mathfrak{h}$ is the Cartan subalgebra of $\mathfrak{sl}_n$, and the commutation relations between the various generators are:
$$[x,x']=0,\quad [y,y']=0,\quad gxg^{-1}=g(x), \quad gyg^{-1}=g(y),$$
$$[x,y]=(x,y)-c\sum_{s\in \mathcal{S}}(\alpha_s,x)(\alpha_s^{*},y)s,$$
for any $x\in \mathfrak{h}^{*}$, $y\in \mathfrak{h}$, $g\in S_n$. Here $\mathcal{S}$ denotes the set of all reflections in $S_n$ and $\alpha_s$ is the equation of the reflecting hyperplane of $s\in \mathcal{S}$.
\end{definition}

The polynomial representation also makes sense for rational Cherednik algebras, and its representation space is $M_{c}(n)=\mathbb{C}[\mathfrak{h}]$. The symmetric group acts naturally, $\mathfrak{h}$ acts by multiplication operators and elements of $\mathfrak{h}^{*}$ act by Dunkl operators:
$$ 
D_y=\partial_y-c\sum_{s\in \mathcal{S}}\frac{(\alpha_s,y)}{\alpha_s}(1-s)
$$
As a generalization of this construction, one can consider the standard module: 
$$
M_{c}(\lambda)=\tau_{\lambda}\otimes \mathbb{C}[\mathfrak{h}],
$$ 
where $\lambda$ is any partition and $\tau_{\lambda}$ is the corresponding irreducible representation of $S_n$. It is well-known that $M_{c}(\lambda)$ has a unique simple quotient $L_{c}(\lambda)$.

\subsection{Finite-dimensional representations}

It turns out that the representation theory of the rational Cherednik algebra depends crucially on the parameter $c$. For example, we have the following classification of finite-dimensional representations.

\begin{theorem}[\cite{BEG}] 
\label{def Lmn}
The algebra $\mathbf{H}_c$ only has finite-dimensional representations if $c = \frac mn$ for some $\gcd(m,n)=1$, in which case it has a unique irreducible representation $$L_{\frac mn}=L_{\frac mn}(n).$$
Furthermore (if $m>0$), one has $$\dim L_{\frac mn}=m^{n-1}, \qquad \dim (L_{\frac mn})^{S_n}=\frac{(m+n-1)!}{m!n!}.$$
\end{theorem}

The representation $L_{\frac{m}{n}}$ is canonically graded and carries a grading-preserving action of $S_n$. In particular, it is a representation of $S_n$, so we can define its Frobenius character:
$$
\ch L_{\frac mn} = \frac{1}{n!}\sum_{\sigma\in S_n}\Tr_{L_{\frac mn}}(\sigma)p_1^{k_1}\ldots p_r^{k_r}
$$
where $p_i$ are power sums, and $k_i$ is the number of cycles of length $i$ in the permutation $\sigma$. The Frobenius character makes sense for any representation of $S_n$, and in particular the Frobenius character of the irreducible $\tau_\lambda$ equals the Schur polynomial $s_{\lambda}$.

\begin{theorem}[\cite{BEG}]
The graded Frobenius character of $L_{\frac mn}$ equals
$$
\ch_{q} L_{\frac mn}=\frac{q^{-\frac{(m-1)(n-1)}{2}}}{[m]_{q}}\phi_{[m]}(h_{n}),
$$
where $[m]_{q} = \frac {1-q^{m}}{1-q}$ and $\phi_{[m]}:\Lambda\to \Lambda$ is the homomorphism defined by $\phi_{[m]}(p_k)=p_k\frac{1-q^{km}}{1-q^{k}}$.
\end{theorem}

For $m=n+1$, Gordon observed a close relation between the representation $L_{\frac{n+1}{n}}$ and Haiman's work. Gordon constructs a certain filtration on $L_{\frac{n+1}{n}}$ and proves the following result.

\begin{theorem}[\cite{Gordon}]
\label{thm:gordon}
The bigraded Frobenius character of $\gr L_{\frac{n+1}{n}}$ is given by the formula
$$\ch_{q,t} \gr L_{\frac{n+1}{n}}=\nabla e_n.$$
\end{theorem}

In \cite{GORS}, Gordon's filtration was generalized to all finite-dimensional representations $L_{\frac{m}{n}}$ and it was conjectured that the bigraded character $\gr L_{\frac{m}{n}}$ is tightly related to the Khovanov-Rozansky homology of the $(m,n)$ torus knot. In light of the conjectures of \cite{AS}, we formulate the following:

\begin{conjecture}
\label{conj:frob}
The bigraded Frobenius character of $\gr L_{\frac{m}{n}}$ is given by:
$$
\ch_{q,t} \gr L_{\frac{m}{n}} = \TP_{n,m}\cdot 1
$$
where $\TP_{n,m}$ are the transformed DAHA elements of Subsection \ref{sub:fhilb}.
\end{conjecture}

When $m=n+1$, the conjecture follows from Theorem \ref{thm:gordon} and Corollary \ref{pn1} below. Conjecture \ref{conj:frob} is also supported by numerical computations, and it is compatible with some structural properties. For example, the symmetry between the $q$ and $t$ gradings of the character has been proven in \cite{GORS}, and this symmetry is manifest in the operators $\TP_{n,m}$. Moreover, Conjecture \ref{conj:frob} was proved in \emph{loc. cit.} at $t=q^{-1}$, by showing that the knot invariant equals the singly graded Frobenius character, see also Section \ref{sec:CY limit} for details. It was also observed by Gordon and Stafford that: 
$$\ch_{q,t}\gr L_{\frac{m+n}{n}}=\nabla \ch_{q,t} \gr L_{\frac{m}{n}},$$
where $\nabla$ is the operator of \eqref{def:nabla}. This matches with the equality:
\begin{equation}
\label{eqn:nabla}
\TP_{n,n+m} = \nabla \TP_{n,m} \nabla^{-1}
\end{equation}
which follows easily from the definition of $\TP_{n,m}$ in Subsection \ref{sub:fhilb}.

\begin{remark}
The above only deals with the uncolored case, since the representation-theoretic interpretation of colored refined knot invariants has yet to be developed. It is proved in \cite{EGL} that at $t=q^{-1}$ the unrefined $\lambda$-colored invariant of the $(m,n)$ torus knot is given by the character of the infinite-dimensional irreducible representation $L_{\frac{m}{n}}(n\lambda)$. It would be interesting to define a filtration on $L_{\frac{m}{n}}(n\lambda)$ that matches their character with refined invariants.
\end{remark} 

\subsection{The Gordon-Stafford construction}

Conjecture \ref{conj:frob} is part of a correspondence between representations of the rational Cherednik algebra and coherent sheaves on Hilbert schemes, which we will now discuss. Kashiwara and Rouquier (\cite{KR}) have constructed a quantization of the Hilbert scheme depending on the parameter $c$, such that the category of coherent sheaves over this quantization is equivalent to the category of representations of $\mathbf{H}_{c}$. In characteristic $p$, the analogous construction has been carried out by Bezrukavnikov-Finkelberg-Ginzburg (\cite{BFG}). We will only be concerned with characteristic 0, in which case the initial result of Gordon and Stafford (\cite{GS}) claims the existence of a map:
\begin{equation}
\label{eqn:gs}
D^{b}\text{Rep}(\mathbf{H}_{c}) \longrightarrow D^{b}\text{Coh}(\Hilb_n)
\end{equation}
for all $c$. The category on the left consists of filtered representations (see \cite{GS} for the exact definition) of the rational Cherednik algebra. One may ask about the image of the unique irreducible finite-dimensional representation $L_{\frac mn}$ under the above assignemnt. During our discussions with Andrei Okounkov, the following conjecture was proposed:

\begin{conjecture}
\label{conj:big}

Under the Gordon-Stafford map \eqref{eqn:gs}, $L_{\frac mn}$ is sent to:
\begin{equation}
\label{eqn:gs2}
\CF_{\frac mn} := p_* \left(\CL_n^{S_{m/n}(1)} \otimes \ldots \otimes \CL_1^{S_{m/n}(n)} \right)
\end{equation}
where $p:\Hilb_{0,n} \longrightarrow \Hilb_n$ is the projection map from the flag Hilbert scheme to the Hilbert scheme (see Subsection \ref{sub:fhilb} for the notations), and $\CL_1,\ldots,\CL_n$ are the tautological line bundles.
\end{conjecture}
The flag Hilbert scheme together with the projection $p$ should be understood in the DG sense, see \cite{Negut} for details. In fact, the above conjecture is a particular case of a far-reaching conjectural framework of Bezrukavnikov--Okounkov, concerning filtrations on the derived category of the Hilbert scheme.

To support Conjecture \ref{conj:big}, note that the functor \eqref{eqn:gs} matches the bigraded character of representations with the biequivariant $K-$theory classes of coherent sheaves. Therefore, Conjecture \ref{conj:big} implies that:
$$
\ch_{q,t} \gr L_{\frac{m}{n}} = \left[ p_*\left(\CL_n^{S_{m/n}(1)} \otimes \cdots \otimes \CL_1^{S_{m/n}(n)}\right) \right]
$$
Comparing with \eqref{eqn:cre}, we see that the object in the right hand side is simply $\TP_{n,m} \cdot 1 \in K$. Therefore, Conjecture \ref{conj:big} implies Conjecture \ref{conj:frob}. 

\subsection{Affine Springer fibres}
\label{sec:affine flag}

Yet another geometric realization of $L_{\frac mn}$ is provided by the affine Springer fibres in the affine flag variety. Let us recall that the affine Grassmannian $\Gr_n$ of type $A_{n-1}$ can be defined as the moduli space of subspaces 
$V\subset \BC((t))$ satisfying $t^nV\subset V$ and a certain normalization condition. Similarly, the affine flag variety $\Fl_n$ can be defined as the moduli space of flags of subspaces in $\BC((t))$ of the form
$V_1\supset V_2\supset\ldots\supset V_{n}\supset V_{n+1}=t^{n}V_{1}$ such that $\dim V_i/V_{i+1}=1$. 

Recall that an affine permutation (of type $A_{n-1}$) is a bijection $\omega:\BZ\to \BZ$ such that $\omega(x+n)=\omega(x)+n$ for all $x$ and $\sum_{i=1}^{n}\omega(i)=\frac{n(n+1)}{2}.$
It is well known that $\Fl_n$ is stratified by the affine Schubert cells $\Sigma_{\omega}$ labelled by the affine permutations. The {\em homogeneous affine Springer fiber}: 
$$
\Sigma_{\frac mn} \subset \Fl_n \qquad (\text{resp. } \Sigma^{\Gr}_{\frac mn} \subset \Gr_n) 
$$ 
is defined as as set of flags (resp. subspaces) invariant under multiplication by $t^m$, where, as above, we assume that $\gcd(m,n)=1$. It is known to be a finite-dimensional projective variety \cite{GKM,KL,LS} and the total dimension of the homology equals (\cite{Hikita,LS}):
$$
\dim H^{*}(\Sigma_{\frac mn})=m^{n-1}, \qquad \dim H^{*}(\Sigma^{\Gr}_{\frac mn})=\frac{(m+n-1)!}{m!n!}.
$$
The similarity between this equation and Theorem \ref{def Lmn} suggests a relation between $L_{\frac mn}$ and the homology of $\Sigma_{\frac mn}$. Indeed, in \cite{OY,VV} the authors constructed geometric actions of the DAHA and trigonometric / rational Cherednik algebras on the space $H^{*}(\Sigma_{\frac mn})$ equipped with certain filtrations. In all these constructions, the spherical parts of the corresponding representations can be naturally identified with $H^{*}(\Sigma^{\Gr}_{\frac mn})$, also equipped with certain filtrations. It is important to mention that the homological grading on $H^{*}(\Sigma_{\frac mn})$ {\em does not} match the representation-theoretic grading on $L_{\frac mn}$. On the other hand, the {\em bigraded} character of $\gr H^{*}(\Sigma_{\frac mn})$
is expected to match the bigraded character of $\gr L_{\frac mn}$ after some regrading, when one takes into account both the geometric filtration on the homology and the generalized Gordon filtration on $L_{\frac mn}$ (see \cite{GORS} for the precise conjecture).

In the next section we give an explicit combinatorial counterpart of this conjecture (Conjecture \ref{conj:shuffle}), which can be explicitly verified on a computer. By Conjecture \ref{conj:frob}, the bigraded character of $\gr L_{\frac mn}$ is given by $\TP_{n,m}\cdot 1$ and hence can be computed combinatorially using \eqref{eqn:fix3}. On the other hand, one can try to compute the bigraded character of $\gr H^{*}(\Sigma_{\frac mn})$ using some natural basis in the homology, which is expected to be compatible with the geometric filtration.

\begin{definition}(\cite{GMV})
We call an affine permutation $\omega$ $m$--stable, if $\omega(x+m)>\omega(x)$ for all $x$.
\end{definition}

\begin{theorem}(\cite{GMV})
The intersection of an affine Schubert cell with the affine Springer fiber $\Sigma_{\frac mn}$ is either empty or isomorphic to an affine space. The nonempty intersections correspond to the $m$-stable affine permutations $\omega$, and the dimension of the corresponding cell in $\Sigma_{\frac mn}$ equals:
$$
\dim \Sigma_{\omega}\cap \Sigma_{\frac mn}=|\left\{(i,j)|\omega(i)<\omega(j), 0<i-j<m, 1\le j\le n\right\}|.
$$
\end{theorem}
In the homology of $\Sigma_{\frac mn}$ one then have a combinatorial basis corresponding to these cells, with the homological gradings given by the above equation. 
In \cite{GMV} the  $m$-stable affine permutations has been identified by an explicit bijection with another combinatorial object, the so-called $m/n$-parking functions, see Section \ref{sec:pf}. The dimension of a cell is translated to a certain combinatorial statistics $\dinv$ on parking functions, which has been obtained earlier by Hikita in \cite{Hikita}.
We also conjecture that the geometric filtration on the homology is compatible with the basis of cells, and admits an easy combinatorial description as the ``area" of the corresponding parking function. Modulo this conjecture, one can show that the bigraded character of $\gr H^{*}(\Sigma_{\frac mn})$ coincides with the combinatorial expression \eqref{eq:def Fr}.

One can similarly describe the cell decomposition of $\Sigma^{\Gr}_{\frac mn}$, which turns out to coincide with the compactified Jacobian of the plane curve singularity $\{x^m=y^n\}$. The affine Schubert cells in $\Gr_n$ cut out affine cells in $\Sigma^{\Gr}_{\frac mn}$, which can be labeled either by the $m$--stable permutations with additional restrictions
(\cite{GMV}) or by the Dyck paths in the $m\times n$ rectangle (\cite{GM1,GM2}). The dimension of such a cell can be rewritten as $\frac{(m-1)(n-1)}{2}-h_{+}(D)$, where $h_{+}(D)$ is an explicit combinatorial statistic on the corresponding Dyck path $D$
(see Section \ref{sec:catalan}).

\section{Combinatorial consequences}
\label{sec:comb}

\subsection{} In this section we focus on the combinatorial structure of uncolored refined knot polynomials.  By (\ref{eqn:sumsyt}), we have for $\gcd(n,m)=1$:
\begin{equation}
\label{eqn:adriano}
\TP_{n,m}\cdot 1 = \sum_{\lambda \vdash n}c_{n,m}(\lambda) \frac {\widetilde{H}_{\lambda}}{g_\lambda},
\end{equation}
where $c_{n,m}(\lambda)$ is the sum of terms $c_{n,m}(T)$ over all standard Young tableaux $T$ of shape $\lambda$:
\begin{equation}
\label{cnd}
c_{n,m}(T)=  \gamma^n\frac {\prod_{i=1}^{n} \chi_i^{S_{m/n}(i)} (qt\chi_i-1)}{\prod_{i=1}^{n-1}\left(1 - qt \frac {\chi_{i+1}}{\chi_i} \right)} \prod_{1\leq i < j\leq n} \omega^{-1} \left( \frac {\chi_j}{\chi_i} \right) 
\end{equation}
Recall that $\chi_i$ denotes the weight of box $i$ in the standard Young tableau $T$ and the constants $S_{m/n}(i)$ are defined by (\ref{def:smn}). Some of the coefficients $c_{n,m}(\lambda)$ have appeared in various sources: for $m=1$ they are remarkably simple and were computed first in \cite{GH} and later rediscovered in \cite{ORS,Sha}. For general $m$ and small $n$ some of these coefficients were computed in \cite[Section 5.3]{ORS}, \cite{DMMSS} and \cite{Sha}. Although the individual terms $c_{n,m}(T)$ have a nice factorized form, their sums $c_{n,m}(\lambda)$ look less attractive, for example:
$$
c_{7,2}(4,3) = (1-q)^2 (1-t)^2 (1-t^2)(1-t^3)(qt-1)
$$
$$
(q^3 t^3+q^3 t^2-q^3+q^2 t^5+2 q^2 t^4+q^2 t^3-q^2 t+q t^6+q t^5-q t^4-2 q t^3-q t^2+t^7-t^5-t^4-t^3)
$$
For hook shapes of size $n$, the coefficients $c_{n,m}(k,1,\ldots,1)$ are equal to a product of linear factors times a sum of $n$ terms. Explicitly, the following formula was computed in \cite{Pieri} using shuffle algebra machinery:
$$
c_{n,m}(k,1,\ldots,1) =\frac {(1 - q)(1 - t)}{q^nt^n}  \prod_{i=1}^{k-1} (1-q^i) \prod_{i=1}^{n-k} (1-t^i) \left(\sum_{i=0}^{n-1} q^{\sum_{j=0}^{k-1} \left \lfloor \frac {mj+i}n \right\rfloor}t^{\sum_{j=1}^{n-k} \left \lceil \frac {mj-i}n \right\rceil} \right)
$$
for all coprime $m$ and $n$, and all $1\leq k \leq n$. For small $k$, this agrees with the computations in \cite[Section 5.3]{ORS}.
Further, we give a combinatorial interpretation of uncolored refined knot invariants, generalizing the so-called {\em ``Shuffle Conjecture''} of \cite{HHRLU}. In \cite{GH} A. Garsia and M. Haiman introduced a bivariate deformation of Catalan numbers, and in \cite{GHagl} (see also \cite{HCatalan}) it was proved that it can be obtained as a weighted sum over Dyck paths. In \cite{GM1} (see also \cite{GM2}) this weighted sum was reinterpreted as a sum over cells in a certain affine Springer fiber and generalized to the rational case. We conjecture that the rational extension of $q,t-$Catalan numbers is given by the $u=0$ specialization of the refined invariant (Conjecture \ref{conj:catalan}) and thus can be computed as a certain sum over tableaux.
The coefficients of the full $u-$expansion of the refined invariant are given by the generalized Schr\"oder numbers. The combinatorial statistics for these numbers was conjectured in \cite{EHKK} and proved in \cite{HSchroeder}, and the rational extension of these statistics was conjectured in \cite{ORS}. We give a conjectural formula for them in terms of tableaux in Conjecture \ref{conj:schroeder}. 

It was conjectured in \cite{HHRLU} that the vector $\nabla e_n = \TP_{n,n+1} \cdot 1$ can be written as a certain sum over parking functions on $n$ cars, and it was shown that the combinatorial formulas for $q,t-$Catalan and $q,t-$Schr\"oder numbers follow from this conjecture. This combinatorial sum was reinterpreted in \cite{Hikita} as a weigthed sum over the cells in a certain parabolic affine Springer fiber, and a rational extension of the combinatorial statistics of \cite{HHRLU} has been proposed. We conjecture that the symmetric polynomials constructed in \cite{Hikita} coincide with $\TP_{n,m} \cdot 1$. This conjecture is supported by vast experimental data provided to us by Adriano Garsia. 

It has been conjectured in \cite{GORS} that the weighted sums of \cite{Hikita} (also \cite{GM1,ORS})
compute the bigraded Frobenius characters of the finite-dimensional representations $L_{\frac{m}{n}}$ (and their specializations), and the Poincar\'e polynomials of Khovanov-Rozansky homology of torus knots. On the other hand, it has been conjectured in \cite{AS,Ch} that refined knot invariants compute the Poincar\'e polynomials of Khovanov-Rozansky homology. Although all of these conjectures remain open, the ``rational Shuffle Conjecture" (Conjecture \ref{conj:shuffle}) provides a consistency check for them, since its left and right hand side are explicit combinatorial expressions independent of knot homology or filtration on $L_{\frac{m}{n}}$ .

Finally, we use the notion of Tesler matrices introduced in \cite{HTesler} (see also \cite{AGHRS,HTesler2}) to compute the residue \eqref{eqn:resknot}, and thus give an explicit formula for refined knot invariants. We will use this to prove the specialization of the rational Shuffle Conjecture at $t=1$.

\subsection{Generalized $q,t$-Catalan numbers}
\label{sec:catalan}

We define a $m/n$ Dyck path to be a lattice path in a $m \times n$ rectangle from the top left to the bottom right corner, which always stays below the diagonal connecting these two corners. Alternatively, a Dyck path is a Young diagram inscribed in the right triangle with vertices $(0,0),(m,0)$ and $(0,n)$. We denote the set of all $m/n$ Dyck paths by $Y_{m/n}$, and it is well known that: 
$$
|Y_{m/n}|=\frac{(m+n-1)!}{m!n!}.
$$
Given a Dyck path $D$, we define, following \cite{GM1} and \cite{GM2}, the statistic:
$$
h_{+}(D)=\left\{x\in D\ \vline\ \frac{a(x)}{l(x)+1}<\frac{m}{n}<\frac{a(x)+1}{l(x)}\right\}.
$$
We define the $m/n$ rational Catalan number as the following weighted sum over Dyck paths:
$$
C_{n,m}(q,t)=\sum_{D\in Y_{m/n}} q^{\delta_{m,n}-|D|}t^{h_{+}(D)},
$$
where $\delta_{m,n}=\frac{(m-1)(n-1)}{2}.$ The polynomial $C_{n,m}(q,t)$ is symmetric in $m$ and $n$ by construction, and it has been conjectured in \cite{GM2} that it is symmetric in $q$ and $t$ as well. Here we propose strengthening the $q,t-$symmetry conjecture by the following:

\begin{conjecture}
\label{conj:catalan}
The following relation holds:
$$
C_{n,m}(q,t) = (h_n|\TP_{n,m}|1)  = \sum_{\lambda \vdash n} \frac {c_{n,m}(\lambda)}{g_\lambda}
$$ 
where $h_n \in V$ is the complete symmetric function.
\end{conjecture}

Indeed, the matrix coefficient in the right hand side of the above relation is symmetric in $q$ and $t$, since \eqref{cnd} implies:
$$
c_{n,m}(\lambda;q,t) = c_{n,m}(\lambda^t;t,q)
$$
Therefore, Conjecture \ref{conj:catalan} implies that:
$$
C_{n,m}(q,t) = C_{n,m}(t,q)
$$
Note that for $m=n+1$, Conjecture \ref{conj:catalan} follows from the results of \cite{GHagl} and Corollary \ref{pn1}.

In \cite{ORS}, the polynomials $C_{n,m}(q,t)$ have been extended to accommodate the extra variable $u$. Given a Dyck path $D$ and an internal vertex $P$, we define $\beta(P)$ to be the number of horizontal segments of $D$ intersected by the line passing through $P$ and parallel to the diagonal (see Figure \ref{fig:beta}). Let $v(D)$ denote the set of internal vertices of $D$.

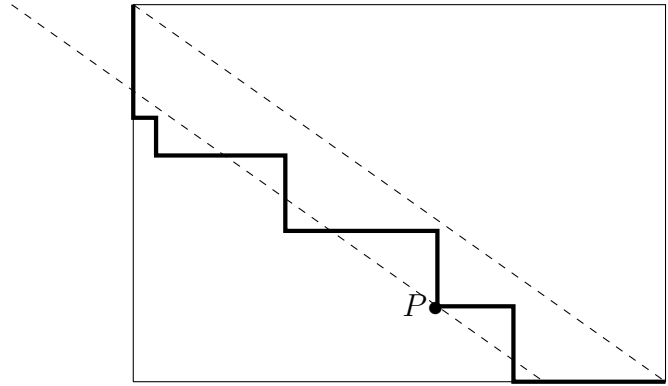
\begin{figure}[ht]
\begin{tikzpicture}
\draw (0,0)--(0,5)--(7,5)--(7,0)--(0,0);
\draw [dashed] (0,5)--(7,0);
\draw [ultra thick] (0,5)--(0,3.5)--(0.3,3.5)--(0.3,3)--(2,3)--(2,2)--(4,2)--(4,1)--(5,1)--(5,0)--(7,0);
\draw (3.81,1.01) node {$P \bullet$};
\draw [dashed] (-1.6,5)--(5.4,0);
\end{tikzpicture}
\caption{Computation of the statistic $\beta(P)$}
\label{fig:beta}
\end{figure}

\begin{conjecture}
\label{conj:schroeder}
The following equation holds:
$$
\sum_{D\in Y_{m/n}} q^{\delta_{m,n}-|D|}t^{h_{+}(D)}\prod_{P\in v(D)}(1-ut^{-\beta(P)}) = 
\TCP_{n,m}(u,q,t)
$$
\end{conjecture}

For $m=n+1$ a similar identity was conjectured in \cite{EHKK} and proved  in \cite{HSchroeder} (see \cite{HBook} and \cite[Section A.3]{ORS} for more details).
 
\subsection{The Rational Shuffle Conjecture}
\label{sec:pf}

The symmetric polynomial $\TP_{n,m}\cdot 1 \in V$ has a combinatorial interpretation.
Let us define a $m/n$ parking function as a function: 
$$
f:\{1,\ldots,m\}\to \{1,\ldots,n\},\ \text{such that }\ |f^{-1}([1,i])|\ge \frac{mi}{n} \quad \forall \ i
$$ 
Alternatively, a parking function can be presented as a standard Young tableau $F$ of skew shape $(D+1^{m})\setminus D$, where $D$ is a $m/n$ Dyck path. Given such a tableau, the function $f$ can be reconstructed by sending each $i$ to the $x$-coordinate of the box labeled by $i$ in the tableau. It is clear that this correspondence is bijective.

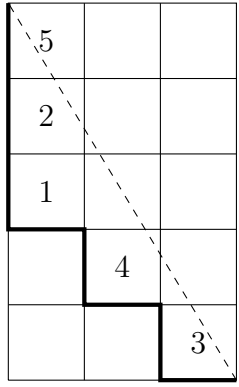
\begin{figure}[ht]
\begin{tikzpicture}
\draw (0,0)--(0,5)--(3,5)--(3,0)--(0,0);
\draw (0,1)--(3,1);
\draw (0,2)--(3,2);
\draw (0,3)--(3,3);
\draw (0,4)--(3,4);
\draw (1,0)--(1,5);
\draw (2,0)--(2,5);
\draw [dashed] (3,0)--(0,5);
\draw [ultra thick] (0,5)--(0,2)--(1,2)--(1,1)--(2,1)--(2,0)--(3,0);
\draw (0.5,4.5) node {$5$};
\draw (0.5,3.5) node {$2$};
\draw (0.5,2.5) node {$1$};
\draw (1.5,1.5) node {$4$};
\draw (2.5,0.5) node {$3$};
\end{tikzpicture}
\caption{A $3/5$ Dyck path and a parking function}
\label{expf}
\end{figure}

Given a box $x=(i,j)$, let us define $r(x)=mn-m-n-mi-nj$. Given a $m/n$ parking function $F$,
define: 
$$
s(F)=|\left\{(x,y)\ :\ x>y \ \text{ such that } \ r(F(y))<r(F(x))<r(F(x))+m\right\}|
$$
and define $s_{\max}(D)$ to be the maximum of $s(F)$ over all parking functions $F$ constructed on the Dyck path $D$.
Let
$$
\dinv(F)=s(F)+h_{+}(D)-s_{\max}(D).
$$
Finally, define the descent set of $F$ by:
$$
\Des(F)=\{x\ :\ r(F(x))>r(F(x+1))\}
$$
Let $\pf_{m/n}$ denote the set of all $m/n$ parking functions. The following symmetric function has been independently constructed in \cite{armstrong} and \cite{Hikita}:
\begin{equation}
\label{eq:def Fr}
\Fr_{n,m}=\sum_{F\in \pf_{m/n}}q^{\delta_{m,n}-|D|}t^{\dinv(F)}Q_{\Des(F)},
\end{equation}
where $Q_{\Des(F)}$ is the Gessel quasisymmetric function \cite{Gessel,HHRLU} associated with the set $\Des(F)$. In \cite{Hikita}, it was proved that $\Fr_{n,m}$ specializes to the symmetric function from \cite{HHRLU} when $m=n+1$, and that $\Fr_{n,m}$ computes the Frobenius character of the $S_n$ action in the homology of a certain Springer fiber in the affine flag variety equipped with extra filtration, as in Section \ref{sec:affine flag}. The following conjecture generalizes this fact for all $m$
and it arose during private communication between the first author and Adriano Garsia:

\begin{conjecture}
\label{conj:shuffle}
The following identity holds:
$$
\Fr_{n,m}=\TP_{n,m}\cdot 1 
$$
\end{conjecture}

When $m=n+1$, it follows from Corollary \ref{pn1} that this conjecture specializes to the main conjecture of \cite{HHRLU}. At the suggestion of Adriano Garsia, we give a constant term formula for $\TP_{n,m}\cdot 1$, which is related to the SYT formula of \eqref{eqn:adriano} in the same way as formula \eqref{eqn:integral} is related to formula \eqref{eqn:sumsyt}. The following formula follows from \cite{Negut}:
\begin{equation}
\label{eqn:constant}
\TP_{n,m}\cdot 1 = \int
\frac {\Psi_n(e(qtz_1)\cdots e(qtz_n)) \prod_{i=1}^n z_i^{S_{m/n}(i)} }{\left(1 - qt\frac {z_{2}}{z_{1}}\right) \cdots  \left(1 - qt\frac {z_{n}}{z_{n-1}}\right)} \prod_{1\leq i < j\leq n} \omega \left( \frac {z_i}{z_j} \right)  \frac {dz_1}{2\pi i z_1}\cdots \frac {dz_n}{2\pi i z_n} 
\end{equation}
where $\Psi_n$ is the map of \eqref{eqn:sloc} and $e(z) = \sum_i (-z)^i e_i$. The above integral goes over contours that surround $0$ and $\infty$, with $z_1$ being the innermost and $z_{n}$ being the outermost contour. One can compute the above residues in $z_i$ and produce a sum of symmetric functions indexed by certain matrices of natural numbers. We will show how to do this in the slightly simpler case of the super-polynomials $\TCP_{n,m}$.

\subsection{The $m=n+1$ case}

For $m=n+1$, one can prove that the rational $q,t$--Catalan numbers agree with the $q,t$--Catalan numbers defined in \cite{GH}, and Conjecture \ref{conj:shuffle} agrees with the ``Shuffle conjecture'' of \cite{HHRLU}.
The following proposition follows from the results of \cite{GH}, but we present its proof here for completeness. 

\begin{proposition} 
One has the identity: $\TP_{0,1}(p_n)=e_n$.
\end{proposition}

\begin{proof}
It is enough to prove this above in any $V_N$, since $V$ is the inverse limit of these vector spaces. 
Let us recall that:
$$
\overline{P}^{N}_{0,k}=t^{-k(N-1)} \left(P^{N}_{0,k}-\frac {t^{kN}-1}{t^k-1}\right)
$$
If we keep $N$ finite but make the change of variables $\ph$ of Theorem \ref{thm:fixed}, we obtain the operator:
\begin{equation}
\label{p01}
\TP^N_{0,1}=\ph \circ \overline{P}^{N}_{0,1} \circ \phh  = t^{-N+1}  \ph \circ \left(\delta_1-\frac {t^{N}-1}{t-1}\right) \circ \phh
\end{equation}
where $\delta_1$ is the operator of \eqref{delta1}. The operators $\TP^N_{0,1}$ stabilize to $\TP_{0,1}$.
Using (\ref{p01}), we can rewrite the desired identity as:
\begin{equation}
\label{delta 1 of p}
\delta_1(p_n)=\frac{1-t^N}{1-t}p_n+(-1)^{n}\frac{t^N(1-q^n)}{t^n(1-t)}\phh(e_{n}).
\end{equation}
Indeed, $\partial_{q}^{(i)}p_n=p_n+(q^n-1)x_i^{n},$ so by (\ref{delta1})
$$
\delta_1(p_n)=p_n\sum_{i}A_{i}(x)+(q^n-1)\sum_{i}A_{i}(x)x_i^{n}.
$$
Consider the function $F(z) = \prod_{i=1}^{N}\frac{1-zx_i}{1-ztx_i}=\sum_{z=0}^{\infty} z^{n} F_{n}.$
It has the following partial fraction decomposition:
$$F(z)=\frac{1}{t^N}+\frac{t-1}{t^N}\sum_{i=1}^{N}\frac{A_i(x)}{1-tzx_i},$$
hence
\begin{equation}
\sum_{i}A_{i}(x)x_i^{n}=\begin{cases}\frac{1-t^N}{1-t}\qquad n=0 \\ \frac{F_n t^N}{t^{n}(t-1)} \ \quad n>0 ,\\\end{cases}
\end{equation}
Therefore we have
$$
\delta_1(p_n)=\frac{1-t^N}{1-t}p_n+\frac{t^N(1-q^n)}{t^n(1-t)}F_n.
$$
On the other hand, 
$$\ln F(z)=\sum_{i=1}^{N}(\ln(1-zx_i)-\ln(1-ztx_i))=-\sum_{k=1}^{\infty}(1-t^{k})\frac{z^{k}p_{k}}{k},$$
hence
$$F(z)=\varphi^{-1}_{\frac{1}{1-t}}\left[\exp (-\sum_{k=1}^{\infty}\frac{z^{k}p_{k}}{k})\right]=\varphi^{-1}_{\frac{1}{1-t}}\left[\prod_{i}(1-zx_i)\right],$$
and $F_{n}=(-1)^{n}\varphi^{-1}_{\frac{1}{1-t}}(e_{n}).$
\end{proof}

\begin{corollary}
\label{pn1}
The following identities hold: 
$$\TP_{n,1}\cdot 1 =e_n,\ \TP_{n,n+1}\cdot 1=\nabla e_n$$
\end{corollary}
 
\begin{proof}
It follows from \eqref{commutation} that $\TP_{n,1}=[\TP_{0,1},\TP_{n,0}]$, hence:
$$\TP_{n,1}\cdot 1=[\TP_{0,1},\TP_{n,0}]\cdot 1=\TP_{0,1}\TP_{n,0}\cdot 1=\TP_{0,1}\cdot p_n=e_n.$$
The second identity follows from the equation $\TP_{n,n+1}=\nabla \TP_{n,1} \nabla^{-1}$.
\end{proof}

As a corollary, we get the following decomposition of the Garsia-Haiman coefficients:

\begin{proposition}
Given a Young diagram $\lambda$, the following identity holds:
\begin{equation}
\label{n1decomposition}
\Pi_{\lambda}B_\lambda = \frac 1{M} \sum^{\text{SYT }}_{T\text{ of shape }\lambda}c_{n,1}(T),
\end{equation}
where: 
$$\Pi_{\lambda}=\prod_{\square \in \lambda}^{\chi(\square)\neq 1} (1-\chi(\square)), \qquad B_{\lambda}= \sum_{\square \in \lambda} \chi(\square)$$ 
The coefficients $c_{n,1}(T)$ are defined by \eqref{cnd}. 
\end{proposition}

\begin{proof}
It has been shown in \cite[Theorem 2.4]{GH} that the left hand side of (\ref{n1decomposition}) coincides with the coefficient:
$$
g_\lambda \cdot \langle e_n, \widetilde{H}_{\lambda} \rangle_{q,t^{-1}}
$$
By Corollary \ref{pn1}, this is equal to $g_\lambda \cdot \langle\widetilde{H}_{\lambda}|\TP_{n,1}| 1 \rangle_{q,t^{-1}}$, which equals the right hand side of \eqref{n1decomposition} by \eqref{cnd}.
\end{proof}

\subsection{Tesler matrices} 
\label{sub:tesler}

By \eqref{eqn:resknot}, the super-polynomials $\CP_{n,m}$ ultimately come down to computing the residue: 
$$
\TCP_{n,m}(u,q,t) = \left(\res_{z_{n}=0} - \res_{z_{n}=\infty} \right) \cdots \left( \res_{z_1=0} - \res_{z_1=\infty}\right) 
$$
\begin{equation}
\label{eqn:res2}
\frac 1{z_1\cdots z_n} \cdot \frac {\prod_{i=1}^n z_i^{S_{m/n}(i)} \cdot \frac {1-uz_i}{z_i-1} }{\left(1 - qt\frac {z_{2}}{z_{1}}\right) \cdots  \left(1 - qt\frac {z_{n}}{z_{n-1}}\right)} \prod_{1\leq i < j\leq n} \omega \left( \frac {z_i}{z_j} \right) 
\end{equation}
For $m>0$, the above only has residues at $z_i = \infty$. Therefore, let us consider the expansions:
$$
\frac {1 - ux}{x-1} = -1 + (u - 1) \sum_{k\geq 1} x^{-k}, \qquad \omega(x) = 1 + \sum_{k=1} A(k) x^{-k}, \qquad \frac {\omega(x)}{1 - \frac {qt}x} = \sum_{k=1} B(k) x^{-k}
$$
where:
$$
A(k) = -(q-1)(t-1)\frac {q^{k}-t^{k}}{q-t}, \qquad B(k) = \frac {(q^{k+1}-q^{k})-(t^{k+1}-t^{k})}{q-t}
$$
Using these, we can compute the \eqref{eqn:res2} inductively. Take first the residue in the variable $z_1$:
$$
\TCP_{n,m}(u,q,t) = \sum^{x_n^i\geq 0}_{x_{n}^1+\ldots+x_n^n = S_{m/n}(n)} (1-u+u\delta_{x_n^n}^0) \cdot A(x_n^{n-1}) \prod^{x_n^i>0}_{i<n-1} B(x^i_n)
$$
$$
\res_{z_{n}=\infty} \cdots \res_{z_{2}=\infty}  \frac 1{z_1 \cdots z_{n-1}} \frac {\prod_{i=2}^{n} z_i^{S_{m/n}(i) + x_i^1} \cdot \frac {1-uz_i}{z_i-1}}{\left(1 - qt\frac {z_{2}}{z_{1}} \right) \cdots  \left(1 - qt\frac {z_{n}}{z_{n-1}} \right)}\prod_{1\leq i < j\leq n-1} \omega \left( \frac {z_i}{z_j} \right)
$$
Following \cite{HTesler} (see also \cite{HTesler2} and \cite{AGHRS}), we introduce the notion of {\em Tesler matrix}. An upper triangular matrix $X=\{x_j^i \geq 0\}_{1\leq i\leq j\leq n}$ is called a $m/n$ Tesler matrix if it satisfies the following system of equations:
\begin{equation}
\label{eq:tesler}
x_{i}^{i}+\sum_{j>i}x_{j}^{i}-\sum_{j<i}x_{i}^{j} = S_{m/n}(i) \qquad \forall \ i
\end{equation}
We will denote the set of all $m/n$ Tesler matrices by $\Tes_{m/n}$. Taking next the residues in the variables $z_{2},\ldots,z_n$ gives us:
\begin{equation}
\label{eqn:tes1}
\TCP_{n,m}(u,q,t) = \sum_{X\in \Tes_{m/n}} \prod_{1\leq i \leq n}^{x_i^i >0} (1-u) \prod_{1\leq i \leq n-1} B(x_{i+1}^i)\prod^{x_j^i>0}_{i<j-1} A(x_j^i)
\end{equation}
Therefore, the above computes the uncolored knot invariant as a sum of certain simple terms over all Tesler matrices. Note that the whole sum depends very strongly on $n$, while the $m$ dependence is captured only in the equation \eqref{eq:tesler}. However, the superpolynomial $\CP_{n,m}$ is conjecturally symmetric in $m$ and $n$, and this is not manifest from the above formula.

\subsection{Degeneration at $t=1$} 

In fact, $\TCP_{n,m}|_{u=0}$ is a polynomial in $q$ and $t$ with positive coefficients, which is not manifest from \eqref{eqn:tes1} above. This follows from the fact that it is the Euler characteristic of a certain line bundle on the flag Hilbert scheme $\Hilb_{0,n}$, as in Subsection \ref{sub:fhilb}. The higher cohomology groups of this line bundle vanish, and $H^0$ only produces positive coefficients (this vanishing result is outside the scope of this paper and will be presented in a future work). However, we can completely describe this polynomial when $t=1$ (or when $q=1$, since the right hand side of \eqref{eqn:tes1} is clearly symmetric in $q$ and  $t$). 

\begin{theorem}
Conjectures \ref{conj:catalan} and \ref{conj:schroeder} hold for $t=1$ and any coprime $n,m$.
\end{theorem}

\begin{proof}
Let us first remark that any Tesler matrix from $\Tes_{m,n}$ gives rise to a Dyck path in the $n\times m$ rectangle, with horizontal steps $x_{i}^{i}$. Indeed, for any $k\le n$ we have the equation:
$$
S_{m/n}(1)+\ldots+S_{m/n}(k) = \sum_{i=1}^{k} \left(x_{i}^{i}+\sum_{j>i}x_{j}^{i}-\sum_{j<i}x_{i}^{j}\right)=
\sum_{i=1}^{k} \left(x_{i}^{i} +\sum_{j>k}x_{j}^{i}\right)
$$
Therefore:
$$
\sum_{i=1}^{k} x_{i}^{i} \leq S_{m/n}(1)+\ldots+S_{m/n}(k) = \left\lfloor \frac{k m}{n}\right\rfloor.
$$
Any given Dyck path may correspond to many Tesler matrices. However, note that:
\begin{equation}
\label{eqn:eval0}
A(x)|_{t=1}= \delta_x^0, \qquad \qquad B(x)|_{t=1} = q^{x}
\end{equation}
so any summand of \eqref{eqn:tes1} that has some $x_j^i>0$ for $j>i+1$ will vanish. Therefore, the only summands of \eqref{eqn:tes1} that survive are those such that $x_j^i=0$ for all $j>i+1$. Such Tesler matrices will be called \emph{quasi-diagonal}, and the set of quasi-diagonal $m/n$ Tesler matrices will be denoted by $\qTes_{m/n}$. Therefore, \eqref{eqn:tes1} becomes:
\begin{equation}
\label{eqn:dyck1}
\TCP_{n,m}(0,q,1) = \sum_{X\in \qTes_{m/n}} q^{x_n^{n-1}+\ldots+x_2^1} \prod_{1\leq i \leq n}^{x_i^i>0}(1-u)
\end{equation}
The condition $x_{i}^{i}>0$ specifies the corners of a Dyck path, while the sum
$\sum_{i=1}^{n-1}x_{i+1}^{i}=\sum_{i=1}^{n}i(S_{m/n}(i)-x_{i}^{i})$ computes the area between the Dyck path and the diagonal. Therefore:
$$
\TCP_{n,m}(0,q,1) = \sum_{D\in Y_{m/n}} q^{\frac{(m-1)(n-1)}{2}-|D|} (1-u)^{\# \text{ of corners of } D}
$$
This implies Conjecture \ref{conj:schroeder} at $t=1$. When we set $u=0$, we obtain Conjecture \ref{conj:catalan} at $t=1$.

\end{proof}

\subsection{Degeneration at $t=q^{-1}$} 
\label{sec:CY limit}

We will compute the knot invariant at $t=q^{-1}$, and show that it is a $q-$analogue of the $m,n-$Catalan number. This proves Conjecture \ref{conj:catalan} at $t=q^{-1}$.

\begin{proposition}

For any $m$ and $n$ with $\gcd(m,n)=1$, we have:
\begin{equation}
\label{eq:qfactorial}
\TCP_{n,m}(0,q,q^{-1}) = \frac {[m+n-1]!}{[m]!  [n]!},
\end{equation}
where $[k] = \frac {q^{k/2}-q^{-k/2}}{q^{1/2}-q^{-1/2}}$ are the $q-$integers and $[k]!=[1]\cdots [k]$ are the $q-$factorials.

\end{proposition}

\begin{proof}
Let us describe the degeneration of all constructions that we used to the case $t=q^{-1}$, where $q$ and $t$ are the equivariant parameters on $\mathbb{C}^2$. Macdonald polynomials $P_{\lambda}(q,t^{-1})$ will degenerate to Schur polynomials $s_{\lambda}$, hence modified Macdonald polynomials $\widetilde{H}_{\lambda}$ will degenerate to modified Schur polynomials $\ph(s_{\lambda})$. 

As it was explained in \cite{AS}, the case $t=q^{-1}$ corresponds to the classical Chern-Simons theory.
The corresponding knot invariants and operators were widely discussed in the mathematical and physical literature,
see e.g. \cite{K,stevan} for more details. In particular, it is shown in \cite[section 3.4]{stevan} that:
$$\TP_{n,m}(q,q^{-1})=D\TP_{n,0}(q,q^{-1})D^{-1},\ \text{where}\ D=\nabla(q,q^{-1})^{\frac{m}{n}}.$$
Remark that $p_{n}=\sum_{k=0}^{n-1}(-1)^{k}s_{(n-k,1^{k})},$ and:
$$\nabla(s_{(n-k,1^{k})})=q^{\frac{(n-k)(n-k-1)}{2}-\frac{k(k+1)}{2}}s_{(n-k,1^{k})}=q^{\frac {n(n-2k-1)}{2}}s_{(n-k,1^{k})},$$
hence:
$$D(s_{(n-k,1^{k})})=q^{\frac {m(n-2k-1)}{2}}s_{(n-k,1^{k})}=q^{\frac{(m-1)(n-1)}{2}+\frac{n-1}{2}-km}s_{(n-k,1^{k})}.$$
Therefore: 
$$\TP_{n,m}(1)=D(p_n)=q^{\frac{(m-1)(n-1)}{2}}\sum_{k=0}^{n-1}(-1)^{k}q^{\frac {n-1}{2}-km}\ph(s_{(n-k,1^{k})}).$$
By  \cite[Theorem 1.6]{BEG}, this vector coincides with the graded Frobenius character 
of the finite-dimensional representation $L_{m/n}$. One can also check (see e.g \cite{g,GORS} for details) that  
its evaluation is given by the equation \eqref{eq:qfactorial}.
\end{proof}

\end{document}